\documentclass{amsart}
\usepackage{amssymb, latexsym, amsmath, eucal,cite}
\usepackage{graphicx}
\usepackage{tikz}
\usepackage{caption}
\usepackage[labelformat=simple,labelfont={}]{subcaption}

\usepackage{setspace}
\usepackage{rotating}
\usepackage{lscape}
\usepackage[letterpaper, left=2.5cm, right=2.5cm, top = 1in, bottom = 1in]{geometry}

\theoremstyle{plain}
\newtheorem{thm}{Theorem}[section]
\newtheorem{lemma}[thm]{Lemma}
\newtheorem{cor}[thm]{Corollary}

\theoremstyle{definition} 
\newtheorem{defn}{Definition}

\begin{document}

\title{Complementary Regions of Multi-Crossing Projections of Knots}

\author{MurphyKate Montee}
\address{MurphyKate Montee, University of Notre Dame}
\email{mmontee@nd.edu}

\begin{abstract}
An increasing sequence of integers is said to be \emph{universal} for knots if every knot has a reduced regular projection on the sphere such that the number of edges of each complementary face of the projection comes from the given sequence. Adams, Shinjo, and Tanaka have, in a work, shown that $(2,4,5)$ and $(3,4,n)$ (where $n$ is a positive integer greater than $4$), among others, are universal.  In a forthcoming paper, Adams introduces the notion of a multi-crossing projection of a knot.  An \emph{$n$-crossing projection} is a projection of a knot in which each crossing has $n$ strands, rather than $2$ strands as in a regular projection.  We then extend the notion of universality to such knots.  These results allow us to prove that $(1,2,3,4)$ is a universal sequence for both $n$-crossing knot projections, for all $n>2$.  Adams further proves that all knots have an $n$-crossing projection for all positive $n$. Another proof of this fact is included in this paper. This is achieved by constructing $n$-crossing template knots, which enable us to construct multi-crossing projections with crossings of any multiplicity.
 \end{abstract}

\maketitle

\section{Introduction}

Knot theorists have long been interested in \emph{regular projections} of knots, in which all singular points of the projection are 
\emph{regular} crossing points; specifically, there are two strands of the knot which cross each other.  In \cite{Ad}, Adams introduces the concept of 
a triple crossing knot projection, in which every singular point of the projection is a triple crossing, a singular point at which three strands 
cross each other so that each strand bisects the crossing.  In his paper, Adams extends this idea to a \emph{multi-crossing}, as defined below:

\begin{defn} A \emph{multi-crossing} of multiplicity $n$, also called an \emph{$n$-crossing}, is a singular point at which $n$ strands cross each other 
so that each strand passes straight through the crossing, thus bisecting the crossing.  
An $n$-crossing projection of a knot is a projection in which the only singular points are $n$-crossings.
\end{defn}

In \cite{AST}, Adams, Shinjo, and Tanaka  investigate \emph{complementary regions} of regular reduced knot projections.
A complementary region in this sense is an extension of the graph theoretic concept; a regular reduced knot projection can be viewed as a planar $4$-valent 
graph.  Projecting that graph on a sphere partitions the sphere into various faces, called complementary regions.  We can identify these  
faces by the number of edges around each face, saying that a face with $m$ sides around the edge is an $m$-gon.

This paper will further extend the concept of complementary regions to $n$-crossing projections of knots.  An $n$-crossing projection can be seen as 
a planar $2n$-valent graph, which will similarly partition a sphere into faces, which we call complementary regions.

\begin{defn}A strictly increasing sequence of integers $(a_1, a_2, \dots, a_k)$ is \emph{realized} by a knot $K$ if $K$ has 
some projection $P$ so that the complementary regions of $P$ are $a_i$-gons for some $i = 1, \dots, k$. Such a sequence is 
\emph{universal for $n$-crossing projections} if every knot has some $n$-crossing projection which realizes the sequence.
\end{defn}

The paper will rely heavily on the following fact:

\begin{thm} Let $P$ be an $n$-crossing knot projection.  Then there exists a choice of crossing data that makes $P$ the trivial knot.
\end{thm}

\begin{proof}  Pick a point of $P$ and an orientation.  Beginning at $P$, follow the knot along the orientation.  
Consider a crossing.  Choose the first strand that passes through the 
crossing to be the top most strand, choose the second strand that passes through the crossing to be the second highest strand of the crossing, and continue 
until the final pass through the crossing.  Do this for each crossing.  This yields a trivial knot.
\end{proof}

This paper will prove the following theorem:

\begin{thm} For all $n>2$, the sequence $(1, 2, 3, 4)$ is universal for $n$-crossing knot projections.
\end{thm}

The proof of this is broken into several steps.  Section 2 will prove this case for $3$-crossing projections.  Section 3 will generalize the results of Section 2 and 
prove this case for $3n$-crossing projections, for all $n\geq2$.  Section 4 will further generalize to $(3n+2)$ and $(3n+4)$-crossing projections where $n\geq1$, 
and Section 5 will prove the case for $4$-crossing projections.

Notice that the complementary regions of a knot projection do not depend on the crossing data of the knot.  Thus this paper will frequently 
represent knots without explicit crossing data.


\section{3-Crossing Projections}


Consider the knot projection of Figure \ref{fig:3template}.  It realizes the sequence $(1, 3)$ and has a hexagonal 
central region tiled by equilateral triangles. We will show that for any size of central hexagon, there exists a knot projection of this form.

\begin{figure}[h!]
\centering
\scalebox{.7}{\includegraphics{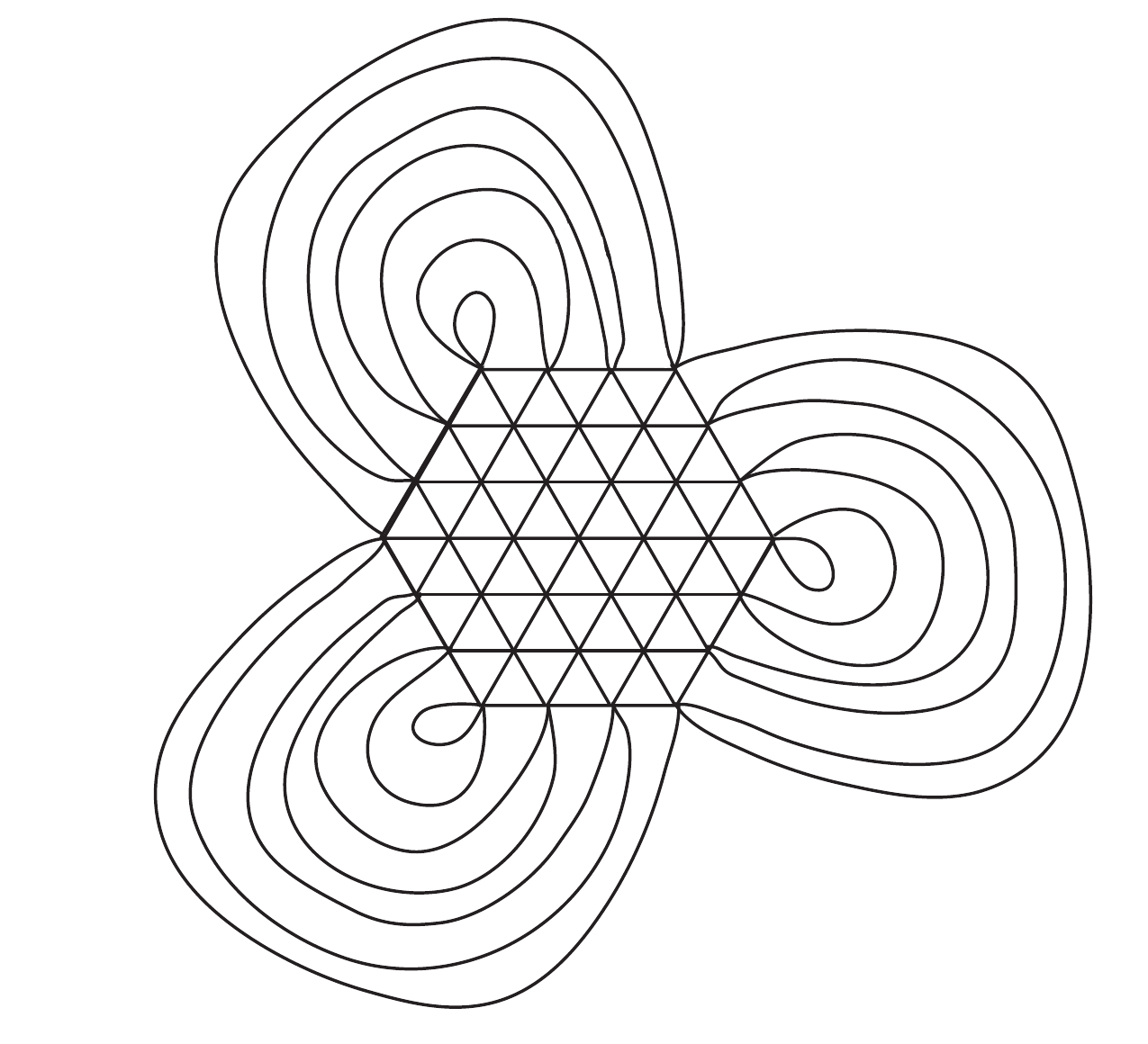}}
\caption{}
\label{fig:3template}
\end{figure}

\begin{lemma}\label{3template}  For any $n>1$, there exists a $3$-crossing knot projection which has a hexagonal central area tiled by 
equilateral triangles, 
with $n$ triangles along each edge of the hexagon, and which realizes the sequence (1, 3). 
\end{lemma}
\begin{proof} 
Consider a hexagon composed entirely of equilateral triangles with $n$ triangles 
along each edge of the triangle.  To ensure that each vertex is a 3-crossing,  let three strands emerge outside each corner of the hexagon, 
and let two strands emerge from the hexagon at each point where two triangles meet along the edge of the hexagon.  
We call these strands \emph{tail edges}.  Thus, there are 
$6(2)(n-1) + 6(3) = 12n + 6 = 3(2)(2n+1)$ tail edges around the edge of the hexagon.  
Consider the upper-right corner of the hexagon.  Label the clockwise-most tail edge $1A$.  Continue 
clockwise around the edge of the hexagon, labeling the strands $2A, 3A, \dots, (2n+1)A, (2n+1)A', (2n)A', \dots, 1A', 1B, 2B, \dots, 
(2n+1)B, (2n+1)B', (2n)B', \dots, 1B', 1C, 2C, \dots, (2n_1)C, (2n+1)C', (2n)C', \dots, 1C'$, as in Figure \ref{fig:generaltemplate}.  

\begin{figure}[h!]
		\centering
		\scalebox{.7}{\includegraphics{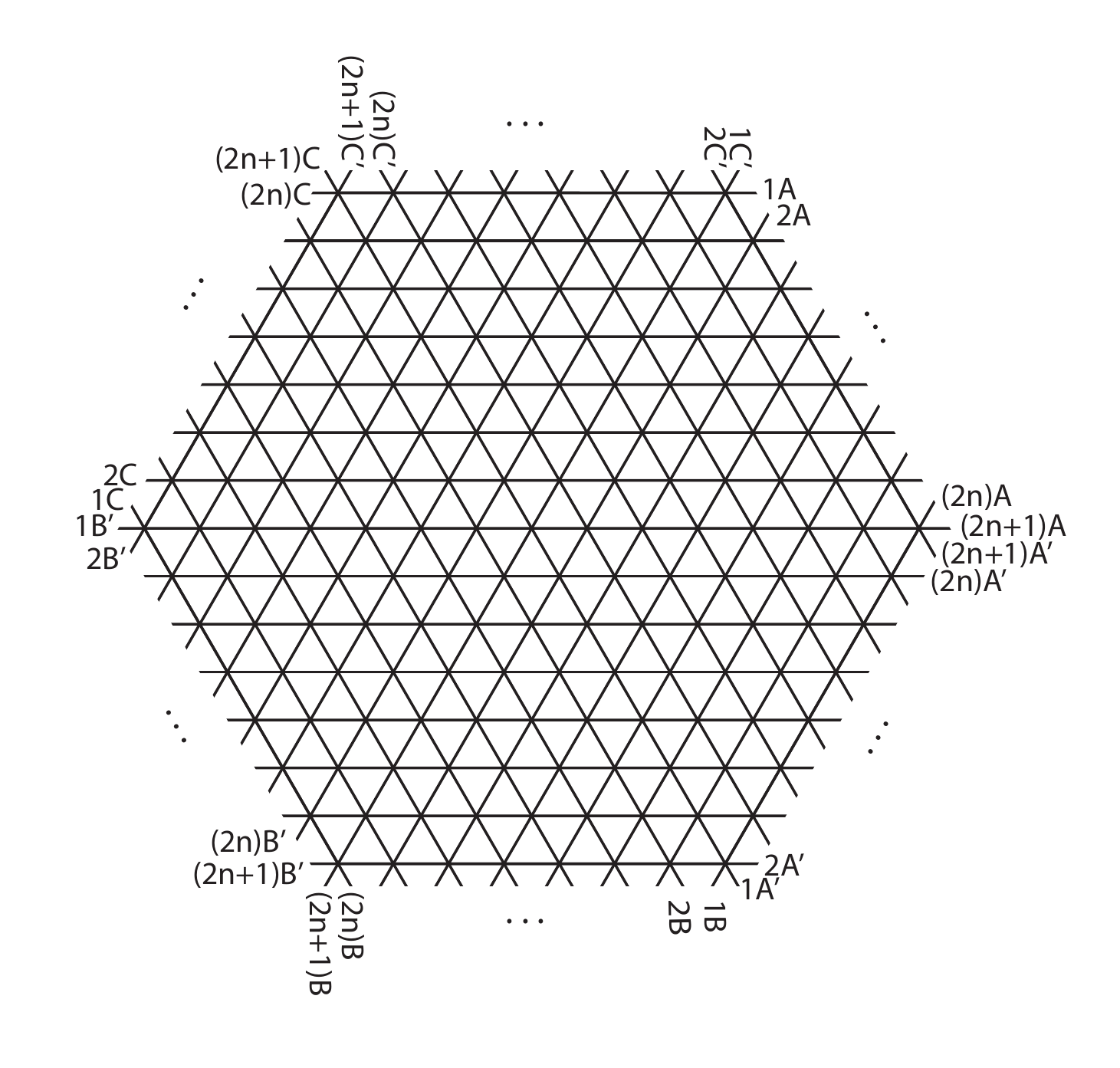}}
		\caption{General triple-crossing template.}
		\label{fig:generaltemplate} 
\end{figure}

On the exterior of the hexagon, we attach 
according to the rules: 
	\begin{align*}
		iA \leftrightarrow iA', \\
		iB \leftrightarrow iB', \\
		iC \leftrightarrow iC',
	\end{align*}
for all $i = 1, \dots, 2n+1$.  Across the hexagon, we get the following relations: 
	\begin{align*}
		1A' \leftrightarrow (2n+1)C, \\
		1B' \leftrightarrow (2n+1)A, \\ 
		1C' \leftrightarrow (2n+1)B. 
	\end{align*}
For odd values of $A, B, C$, we get the following relations: 
	\begin{align*}
		(2i+1)A \leftrightarrow (2(n-i))C, \\
		(2i+1)B \leftrightarrow (2(n-i))A, \\
		(2i+1)C \leftrightarrow (2(n-i))B, 
	\end{align*} 
for $i = 0, \dots, n-1$.  
For odd values of $A', B', C'$, we get the following relations: 
	\begin{align*}
		(2i+1)A' \leftrightarrow (2(n-i)+2)C', \\
		(2i+1)B' \leftrightarrow (2(n-i)+2)A', \\
		(2i+1)C' \leftrightarrow (2(n-i)+2)B', 
	\end{align*}
for $i = 1, \dots, n$.  This gives us relations for all values of $A, A', B, B', C,$ and $C'$.  Choose an orientation.  Following a strand, 
we obtain the sequence of Table \ref{table:3crossingsequence}.

Notice that the first line consists of all $(\rm{odd})C', (\rm{odd})C, (\rm{even})B', (\rm{even})B$, 
the second line consists of all $(\rm{even})A,$ $(\rm{even})A',$ 
$(\rm{odd})B$, $(\rm{odd})B'$, and the last line consists of all $(\rm{odd})A',$ $(\rm{odd})A,$ $(\rm{even})C',$ $(\rm{even})C$.  
Since this sequence therefore contains all 
exterior loops and all straight lines through the hexagon, this is a single strand that covers the entire projection, so the projection is a knot.
\end{proof}

Given this knot projection, 
we can choose orderings on the strands at each crossing so that the knot is the unknot.  We will call an unknot of this type of 
projection a \emph{3-crossing template knot}.  Notice that all 3-crossing template knots are composed of three 
monogons and many triangles.

\begin{thm} The sequence $(1, 2, 3, 4)$ is universal for all 3-crossing knots.
\end{thm}\label{thm:3crossing}
\begin{proof}  Given a knot $K$, consider a regular polygonal projection $P$ of $K$
so that all the line segments composing the knot are horizontal or vertical lines.  
Lay $P$ on a square grid.  We may scale $P$ and the grid as necessary to ensure that the intersections and corners of $P$ 
lie in the center of a grid square; we may further ensure that there are at least two consecutive squares in a single row of our grid 
such that $P$ intersects them as only a single horizontal strand in both squares.
Transpose the knot to the left by half a 
square and down by half a square, and bisect the squares of this grid by parallel diagonal lines with negative slope, creating a triangular grid.   
Note that each corner and intersection of $P'$ is located on a lower left triangle.  Shear this grid and knot $60^\circ$ to 
the right to obtain an 
equilateral triangle grid on which $P$ is overlaid.  Thus, given a sufficiently large template knot, we can 
lay a regular projection of our knot over the template knot so that the corners and intersections of the knot occur on alternating triangles in 
the template knot. 

There are four types of triangles that can occur in this process:  (a) a corner of the original knot, (b) two uncrossing strands of 
the original knot (this occurs up and to the right of a crossing of the original knot), (c) a crossing of the original knot, and (d) a single 
strand of the original knot.  These are depicted in Figure \ref{fig:knotontemplate}.

\begin{figure}[h!]
	\begin{subfigure}{0.3\textwidth}
		\centering
		\scalebox{.3}{\includegraphics{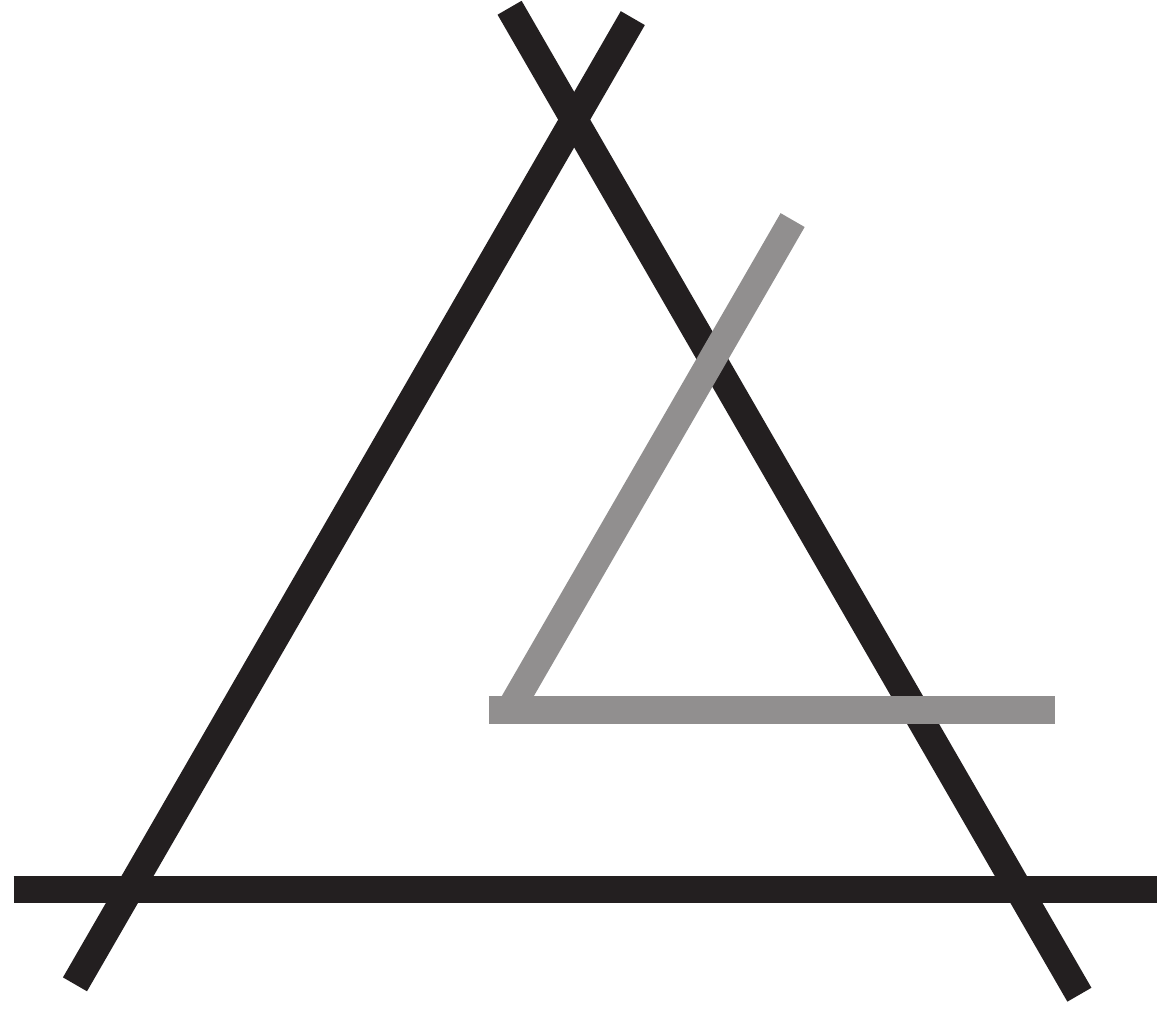}}
		\caption{}
		\label{fig:trianglewithcorner} 
	\end{subfigure}
	\begin{subfigure}{0.3\textwidth}
		\centering
		\scalebox{.3}{\includegraphics{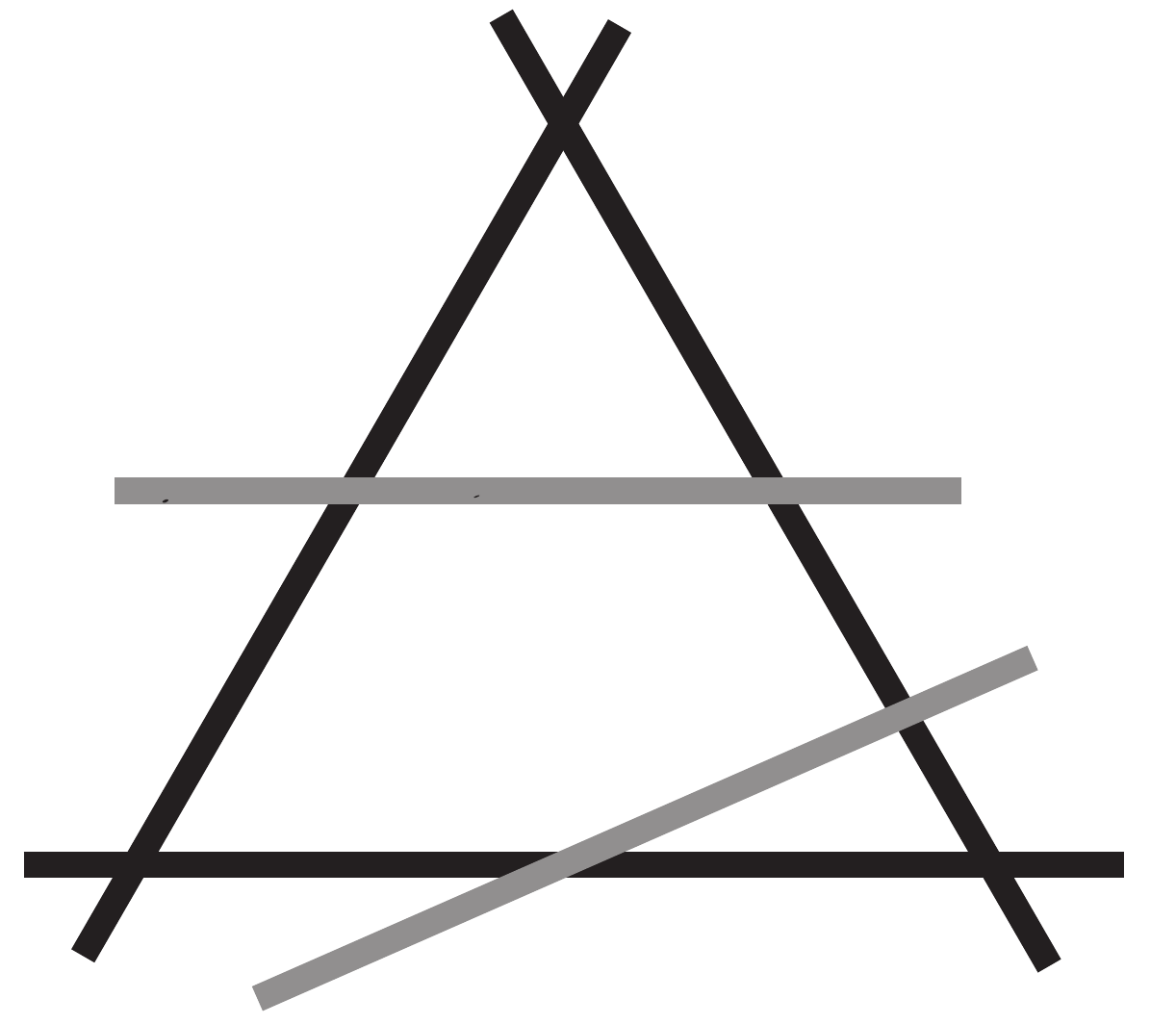}}
		\caption{}
		\label{fig:trianglewithtwouncrossingstrands} 
	\end{subfigure}
	\begin{subfigure}{0.3\textwidth}
		\centering
		\scalebox{.3}{\includegraphics{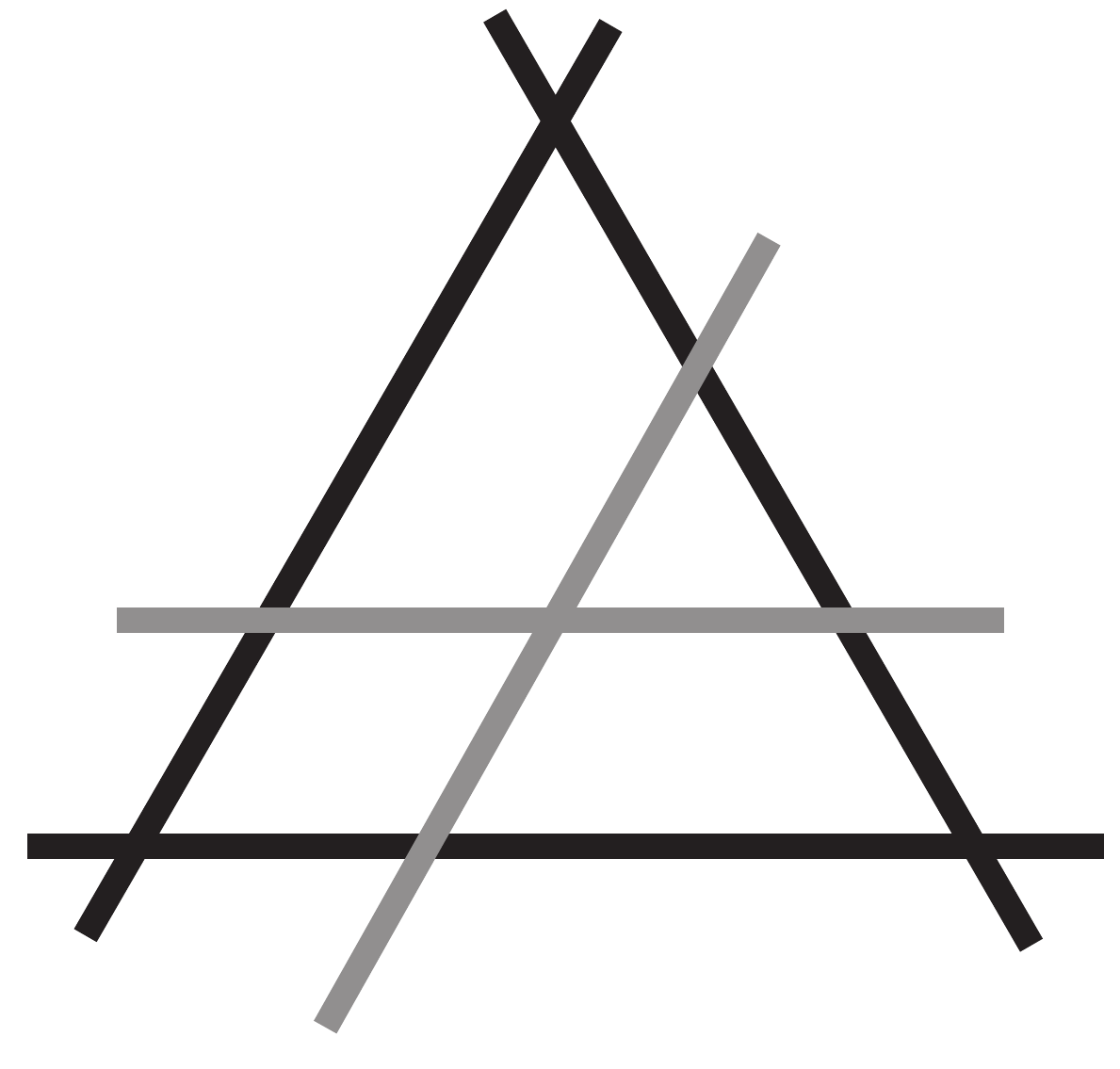}}
		\caption{}
		\label{fig:trianglewithcrossing} 
	\end{subfigure}
	\begin{subfigure}{0.3\textwidth}
		\centering
		\scalebox{.3}{\includegraphics{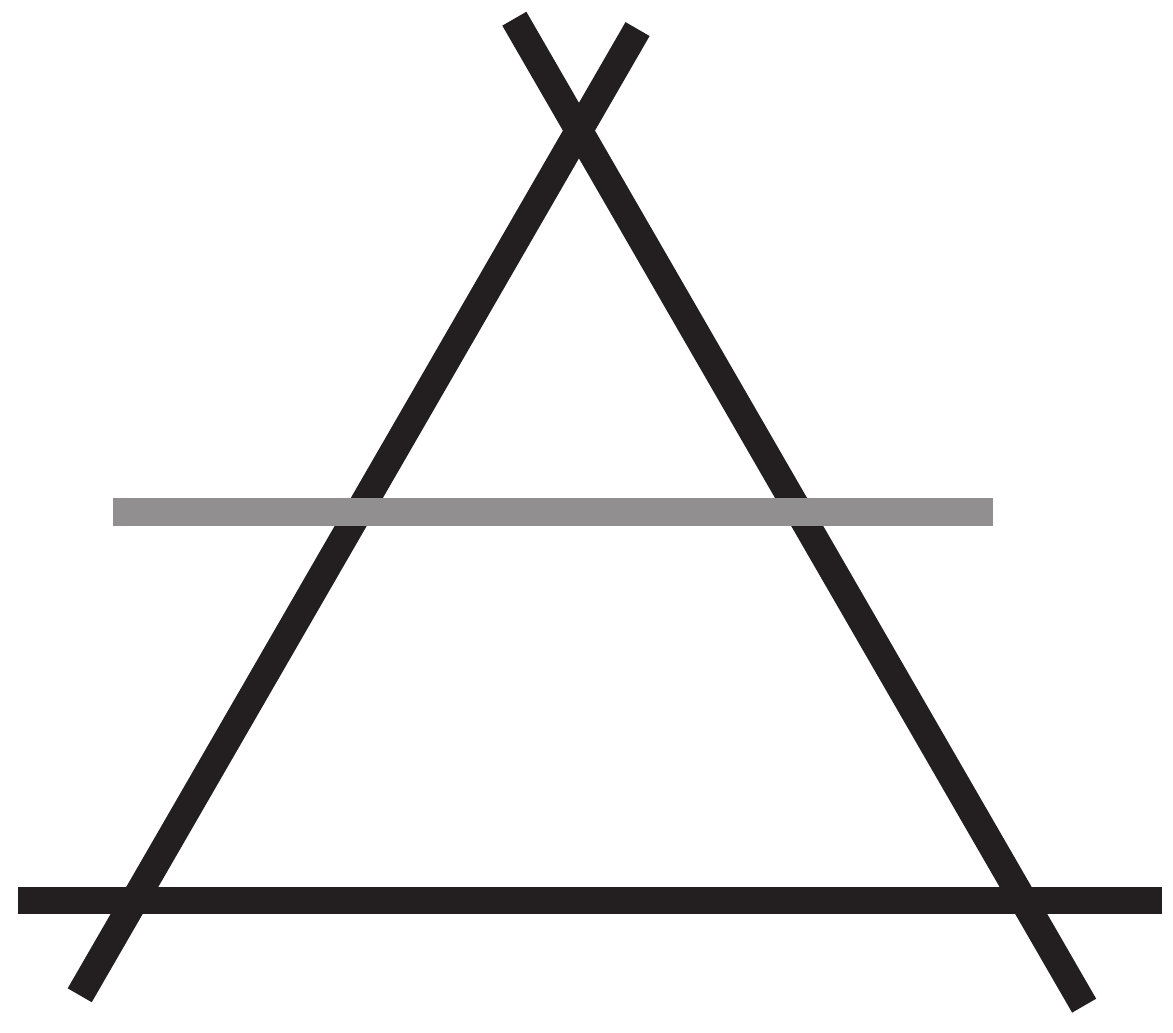}}
		\caption{}
		\label{fig:trianglewithsinglestrand} 
	\end{subfigure}
\caption{}
\label{fig:knotontemplate}
\end{figure}

Since triangles of Type \ref{fig:trianglewithcorner} may be removed by a Reidemeister II move,  
we need not consider triangles of this type.

Take a copy of $P$ and lay it next to $P$ closely enough that all the edges emerge from the same side of a triangle as 
the first knot, but change the crossing data in the copy so that it is the trivial knot.  We will call this 
trivial copy the \emph{doubling knot}, and we will call our first knot the \emph{original knot}. Pick a triangle and perform a 
Reidemeister II move so that the intersections with the doubling knot and the original knot occur precisely at the intersection of the 
original knot and the template knot.  This turns these intersections into 3-crossings.  At a crossing of our original knot with itself, perform a 
Reidemeister II move as in the Figure \ref{fig:trianglewithdoubledcrossing}.  This produces three types of triangles, as shown in Figure \ref{fig:doubledknot}.

\begin{figure}[h!]
	\begin{subfigure}[b]{0.3\textwidth}
		\centering
		\scalebox{.3}{\includegraphics{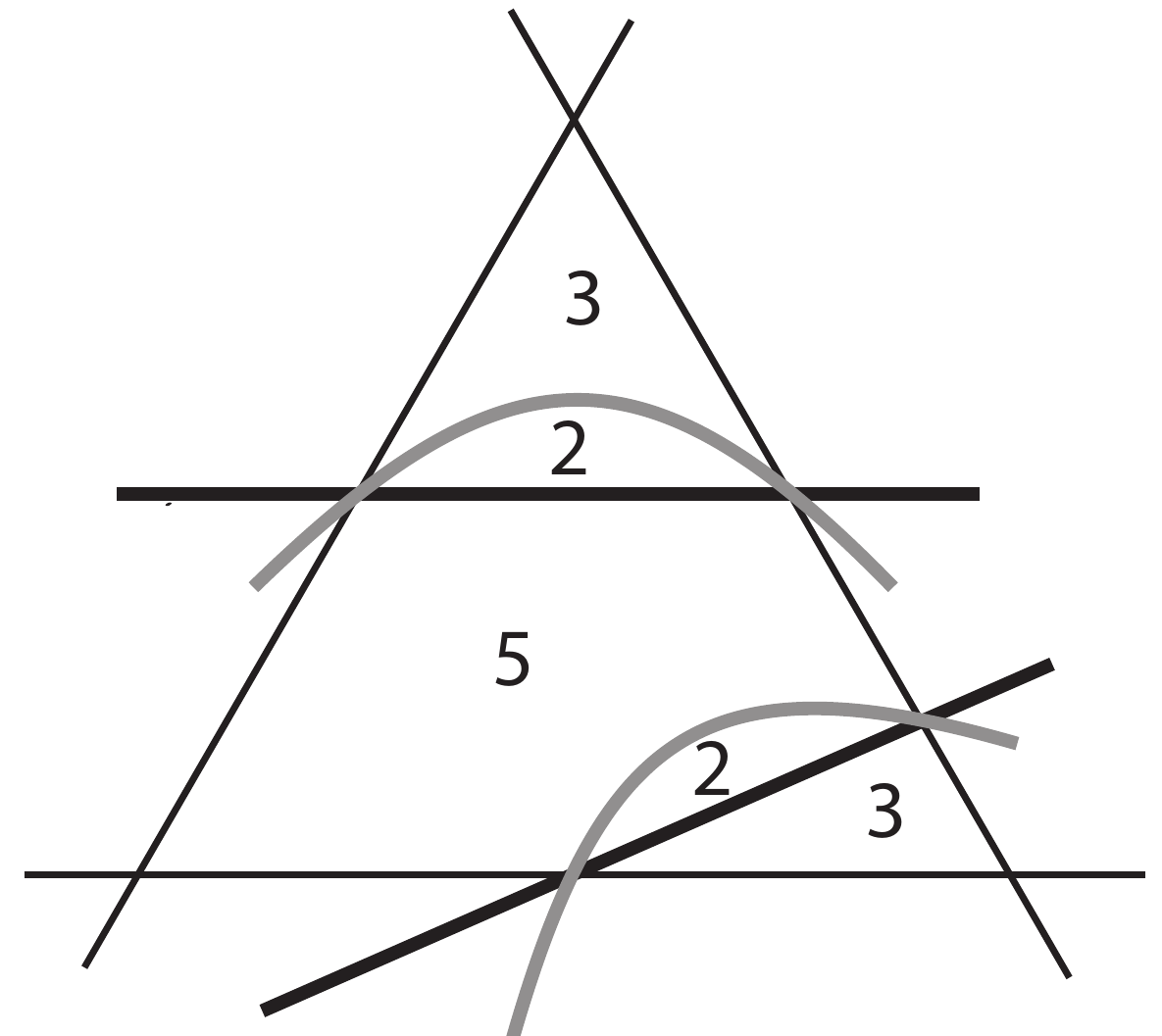}}
		\caption{}
		\label{fig:trianglewithtwouncrossingdoubledstrands} 
	\end{subfigure}
	\begin{subfigure}[b]{0.3\textwidth}
		\centering
		\scalebox{.3}{\includegraphics{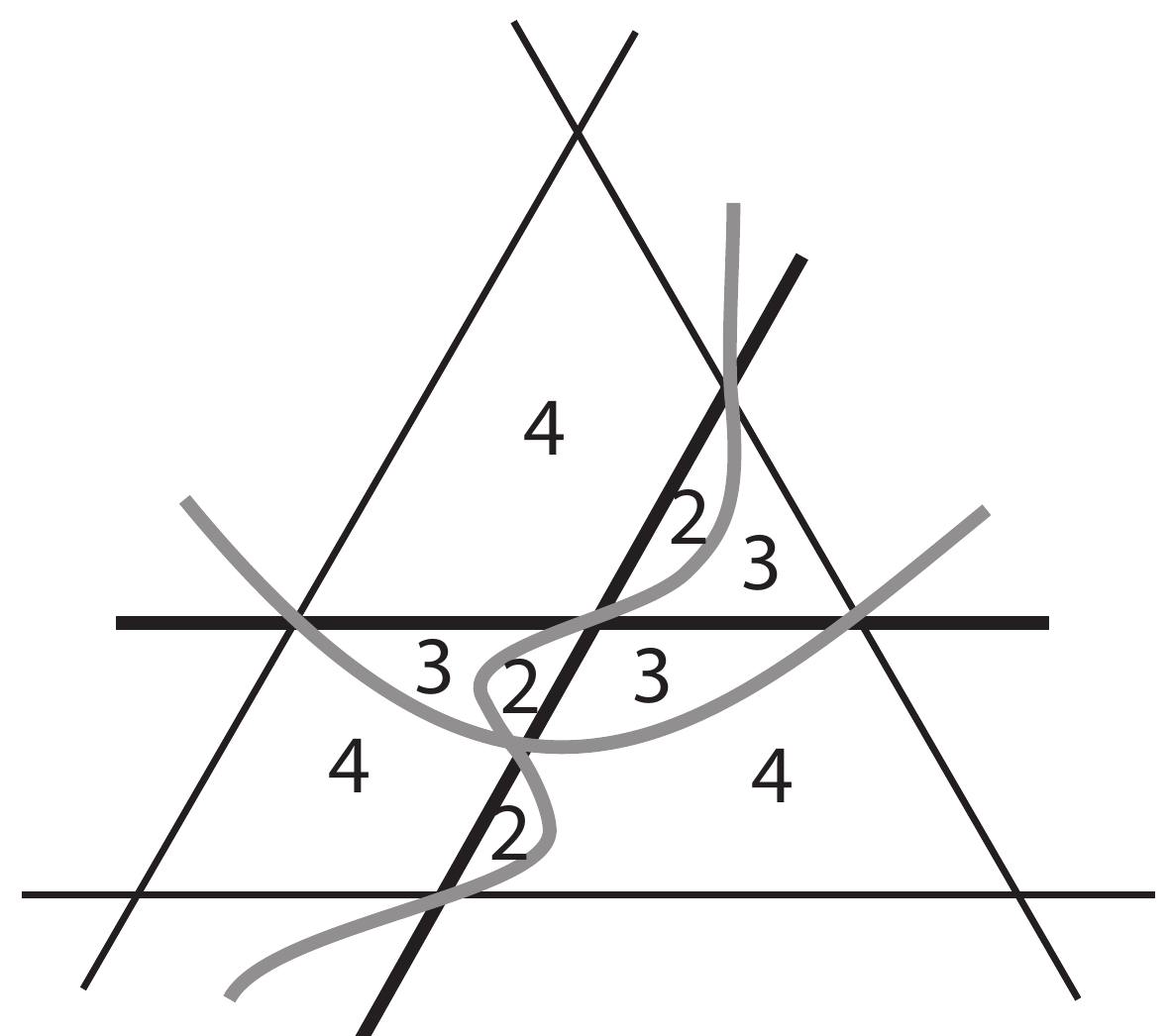}}
		\caption{}
		\label{fig:trianglewithdoubledcrossing} 
	\end{subfigure}
	\begin{subfigure}[b]{0.3\textwidth}
		\centering
		\scalebox{.3}{\includegraphics{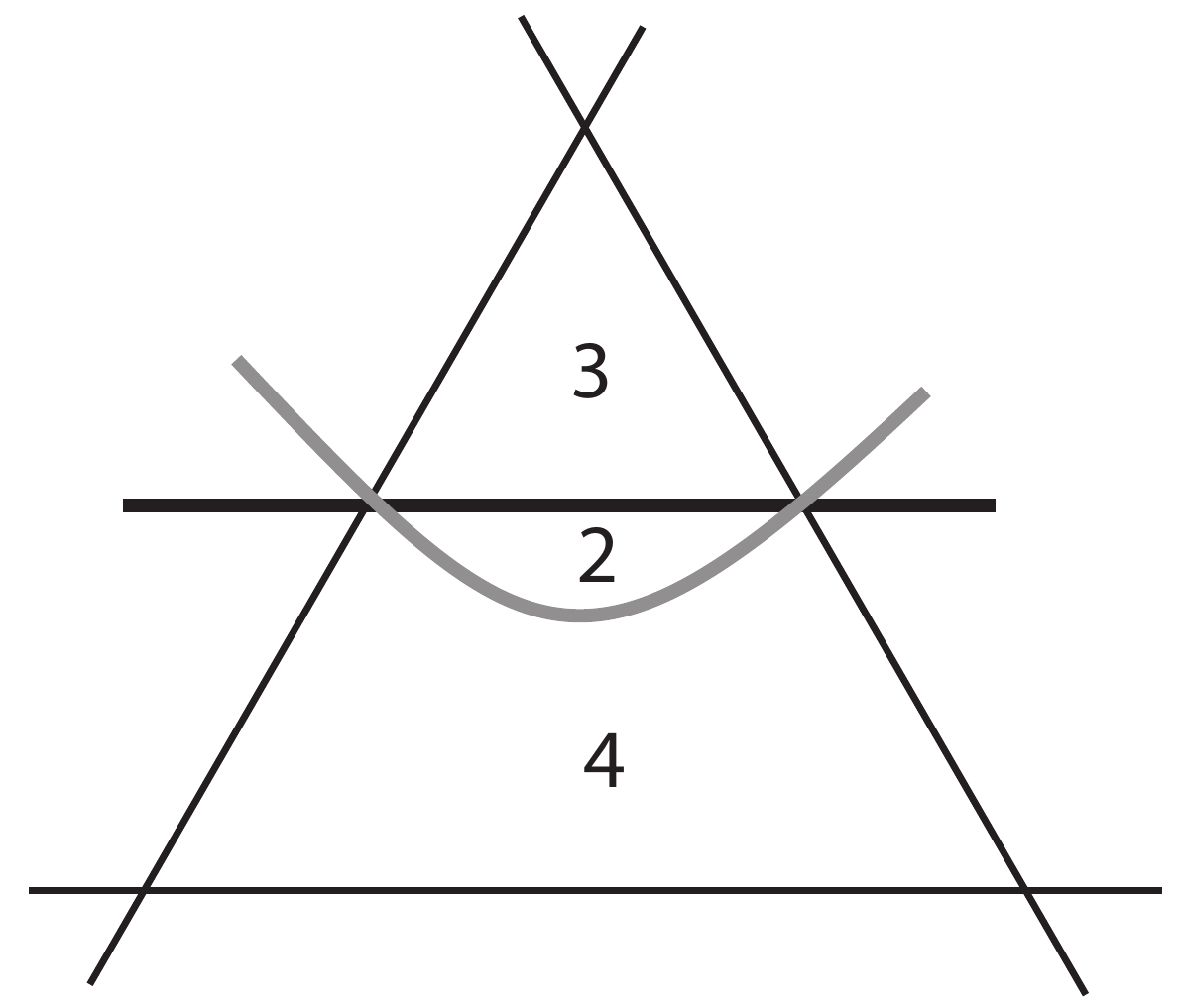}}
		\caption{}
		\label{fig:trianglewithsinglestranddoubled} 
	\end{subfigure}
\caption{(a) Two doubled uncrossing strands. (b) Doubled crossing. (c) Single doubled strand.}
\label{fig:doubledknot} 
\end{figure}

There is only one case which produces a 5-gon.  This can be altered using the method shown in Figure \ref{fig:fixingthe5gon}.

\begin{figure}[h!]
	\begin{subfigure}[b]{0.3\textwidth}
		\centering
		\scalebox{.3}{\includegraphics{trianglewithtwouncrossingdoubledstrands}}
		\caption{}
		\label{fig:fixing5gon1} 
	\end{subfigure}
	\begin{subfigure}[b]{0.3\textwidth}
		\centering
		\scalebox{.3}{\includegraphics{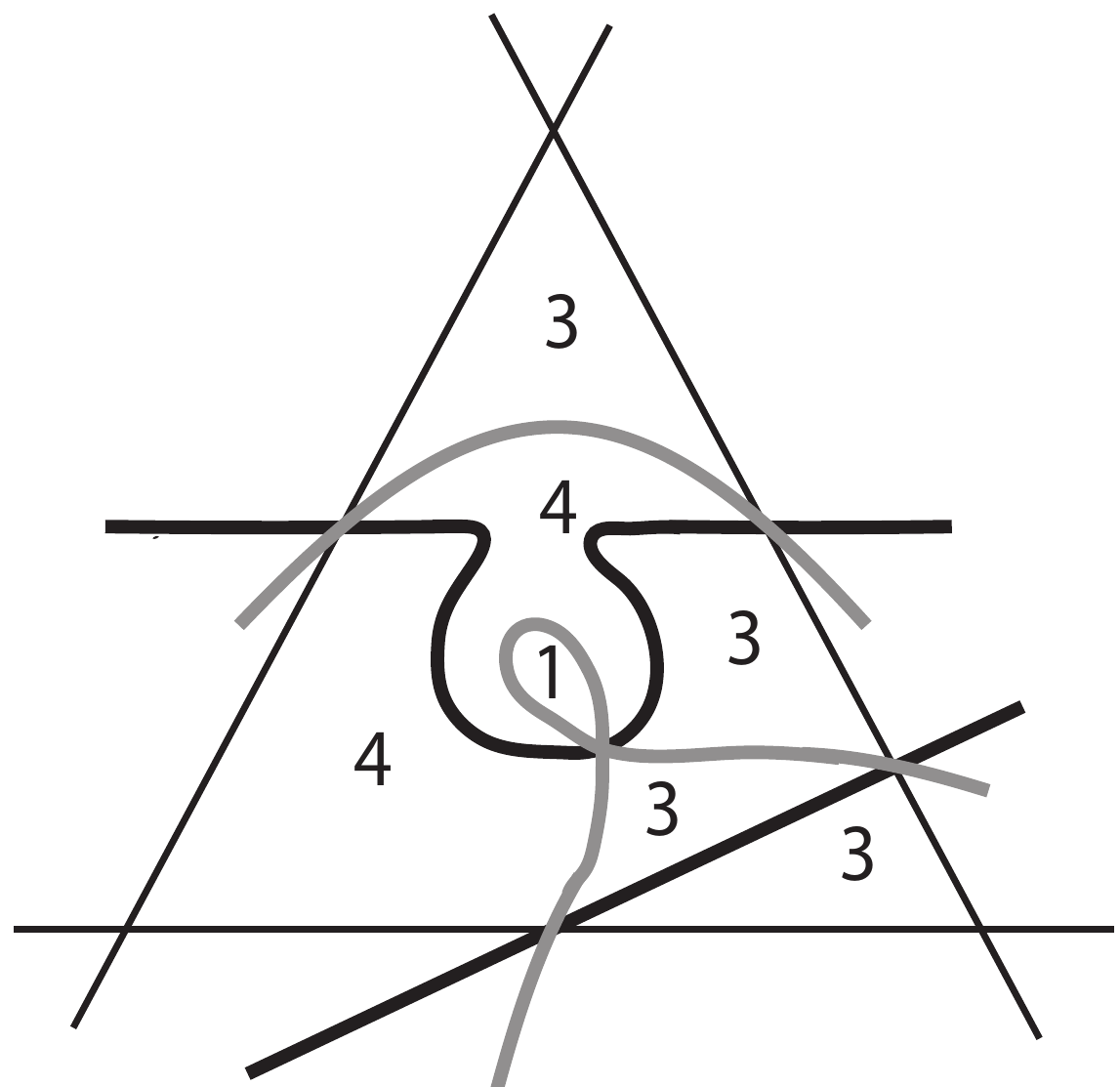}}
		\caption{}
		\label{fig:fixing5gon2} 
	\end{subfigure}
\caption{(a) Two doubled uncrossing strands, creating a 5-gon.  (b) Adding an additional trivial crossing to eliminate the 5-gon.}
\label{fig:fixingthe5gon} 
\end{figure}

Now notice that the method of Reidemeister II may require the doubling knot to traverse the original knot twice before rejoining with itself.  This 
is impossible, though, as we show now.  Bicolor the triangles of the template.
Choose an orientation on the original knot and 
pick a point on the doubling knot. 
We may assume this point is left of the original knot, when following the orientation, and in a black triangle.  
At each intersection of the original knot with itself or the template 
knot, the doubling knot switches to the opposite side.  In a triangle with a crossing, the doubling knot starts to the left, switches to the right, and then returns to the left of the original knot before exiting the triangle.  Thus, at each crossing of the 
original knot with the template, the doubling switches from left to right when exiting a black triangle, and from right to left when exiting a 
white triangle.  Thus the doubling knot must return to its starting position.


Now we need only compose the original knot with the template and with the doubling knot.  Compose the original knot with the doubling knot 
in the manner depicted in Figure \ref{fig:compositionwithdouble}, and compose the original knot with the template knot in the manner depicted 
in Figure \ref{fig:compositionwithtemplate}.

\begin{figure}[h!]
	\begin{subfigure}[b]{0.3\textwidth}
		\centering
		\scalebox{.3}{\includegraphics{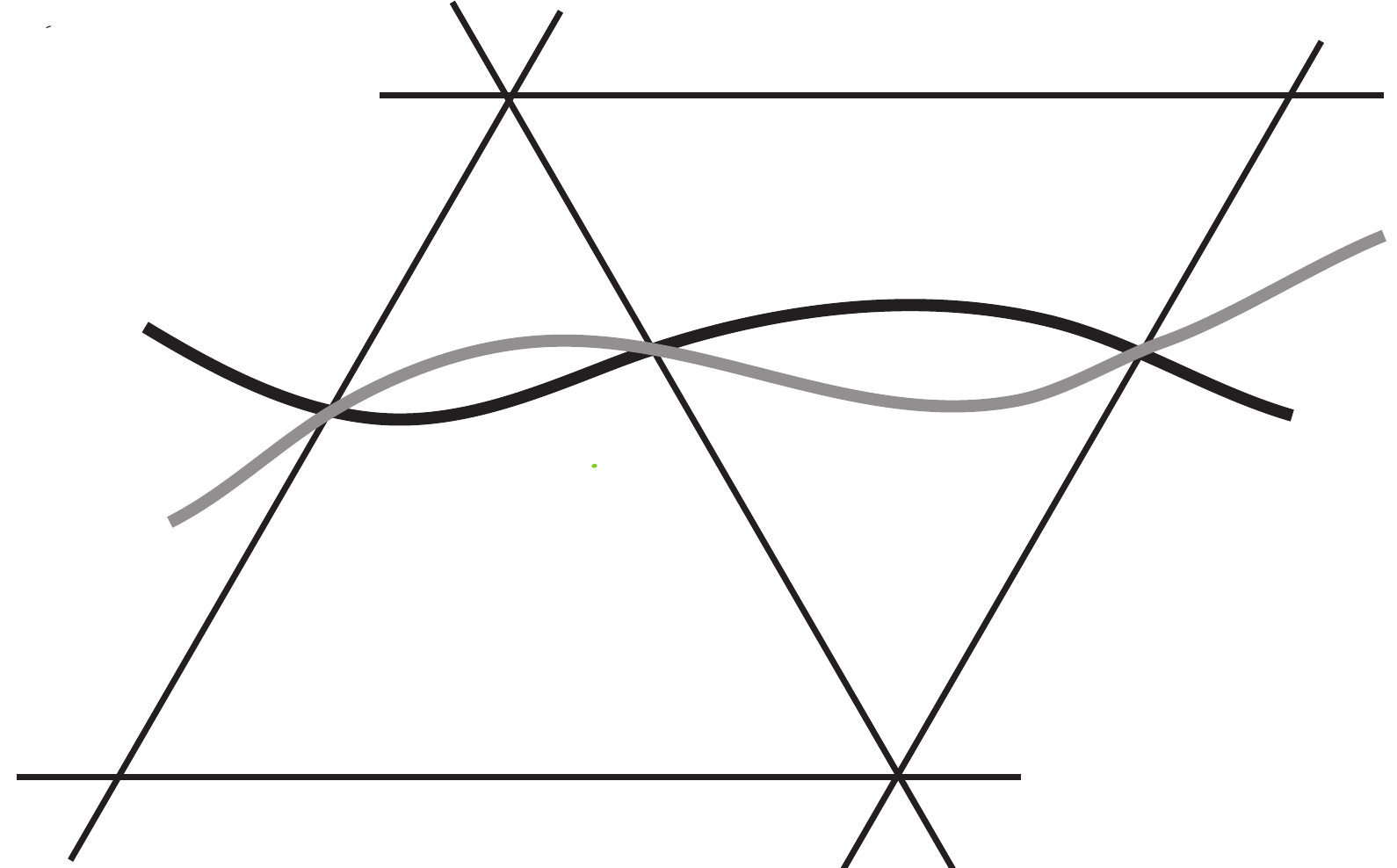}}
		\caption{}
		\label{fig:composewithdouble1} 
	\end{subfigure}
	\begin{subfigure}[b]{0.3\textwidth}
		\centering
		\scalebox{.3}{\includegraphics{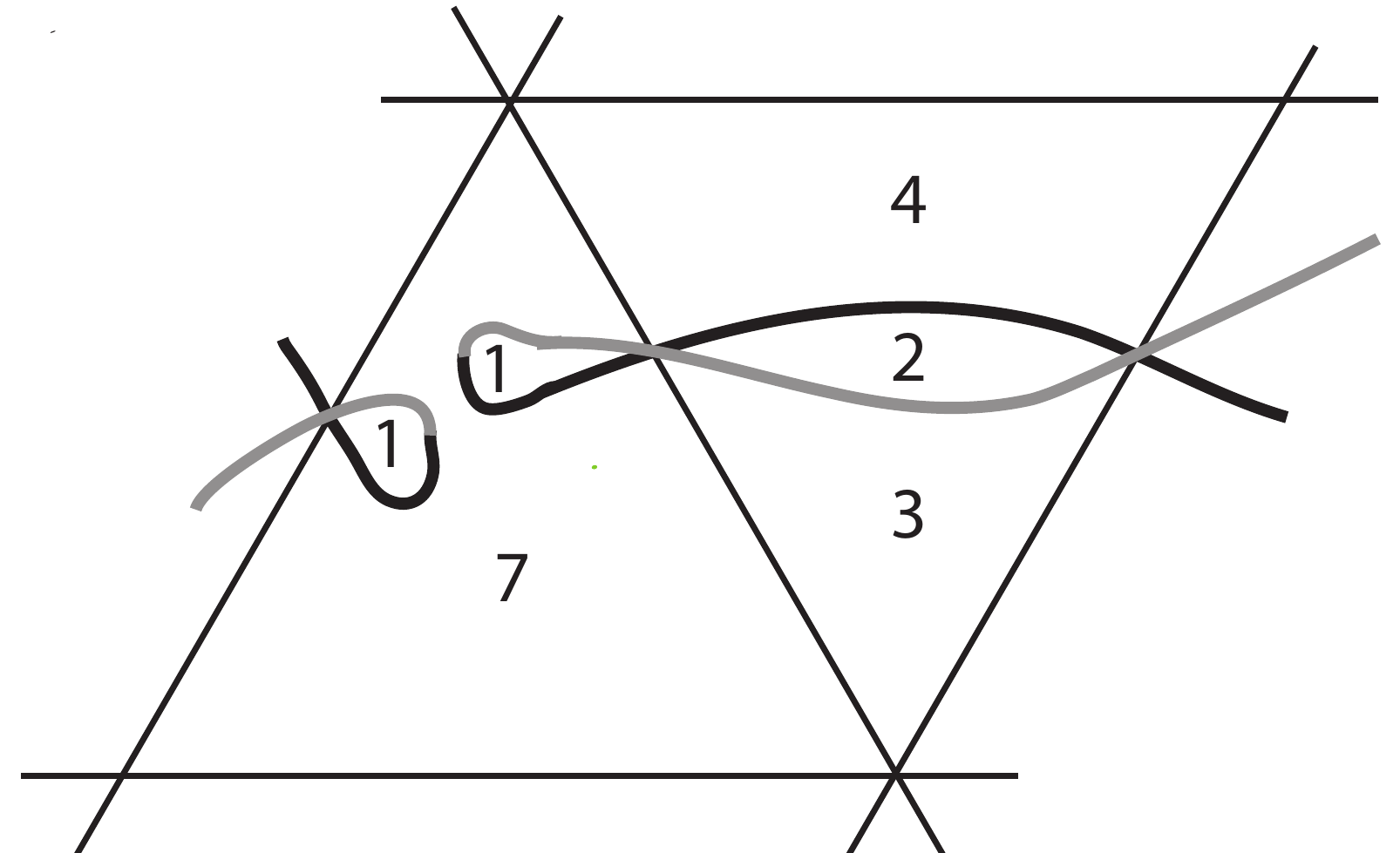}}
		\caption{}
		\label{fig:composewithdouble2} 
	\end{subfigure}
	\begin{subfigure}[b]{0.3\textwidth}
		\centering
		\scalebox{.3}{\includegraphics{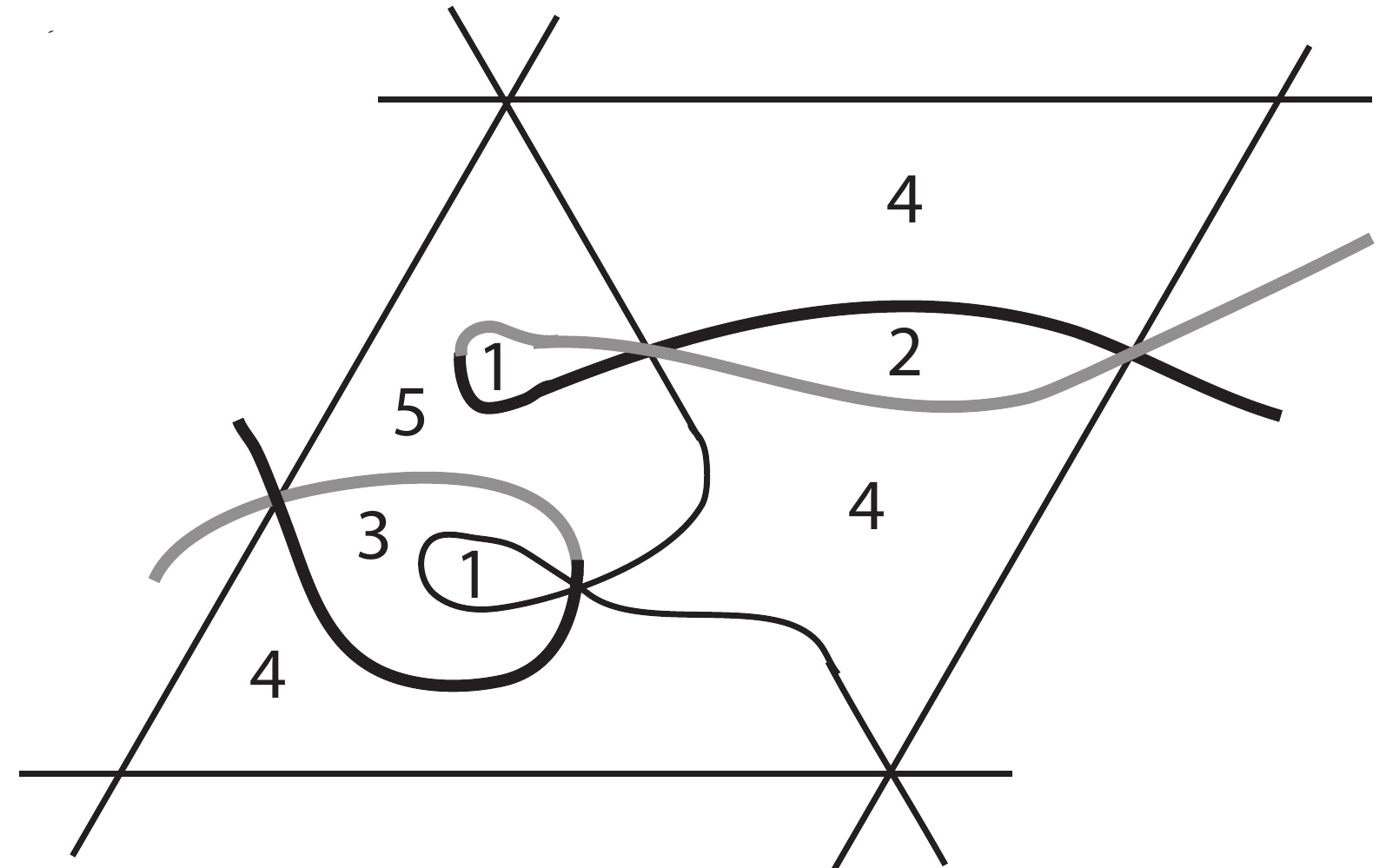}}
		\caption{}
		\label{fig:composewithdouble3} 
	\end{subfigure}
	\begin{subfigure}[b]{0.3\textwidth}
		\centering
		\scalebox{.3}{\includegraphics{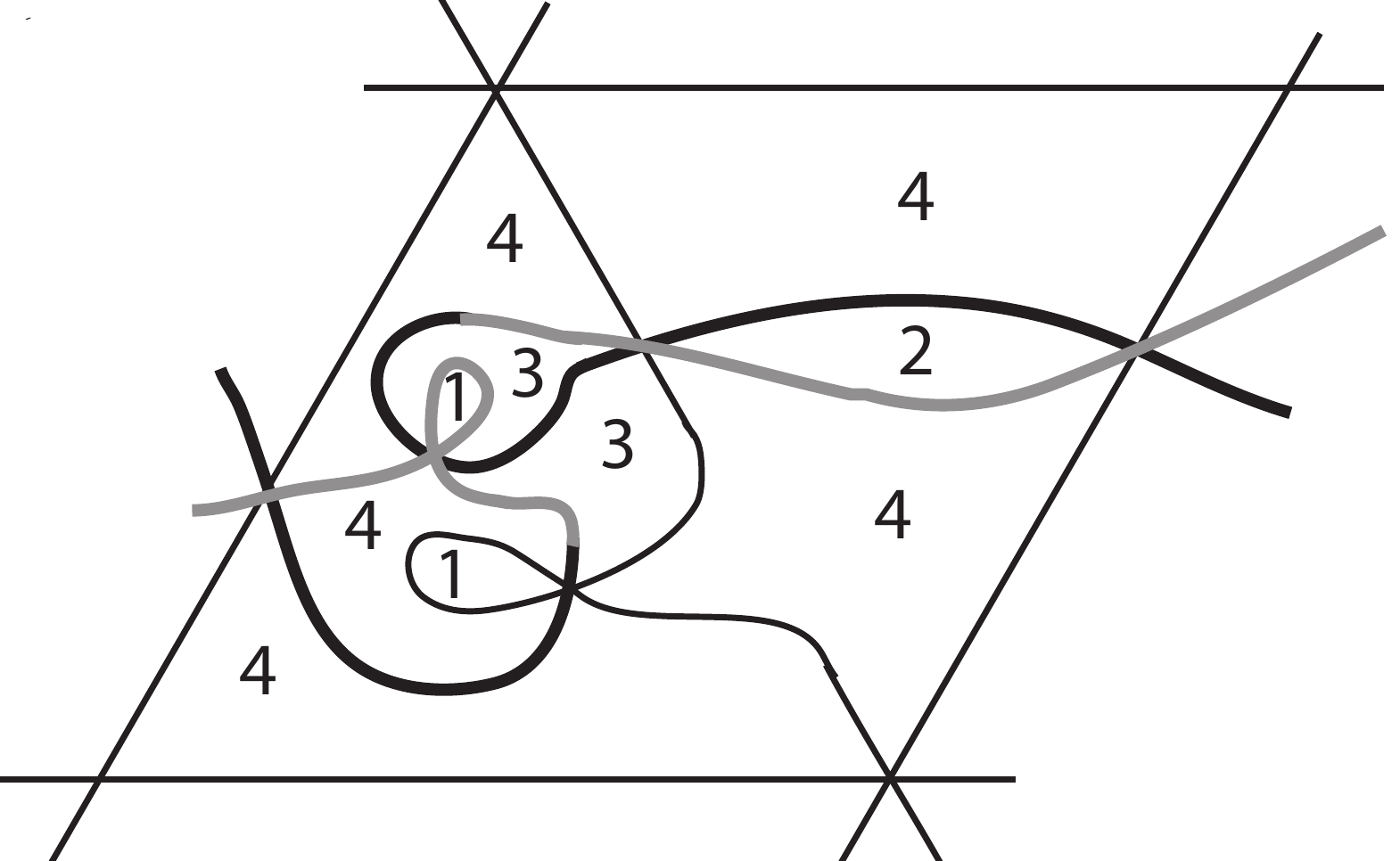}}
		\caption{}
		\label{fig:composewithdouble4} 
	\end{subfigure}
\caption{Composition with the doubling knot. (a) Beginning projection. (b) Composition of the original knot and the doubling knot.  
		(c) Trivial 3-crossing to fix the 7-gon.  (d) Trivial 3-crossing to fix the 5-gon.}
\label{fig:compositionwithdouble} 
\end{figure}

\begin{figure}[h!]
	\begin{subfigure}[b]{0.3\textwidth}
		\centering
		\scalebox{.3}{\includegraphics{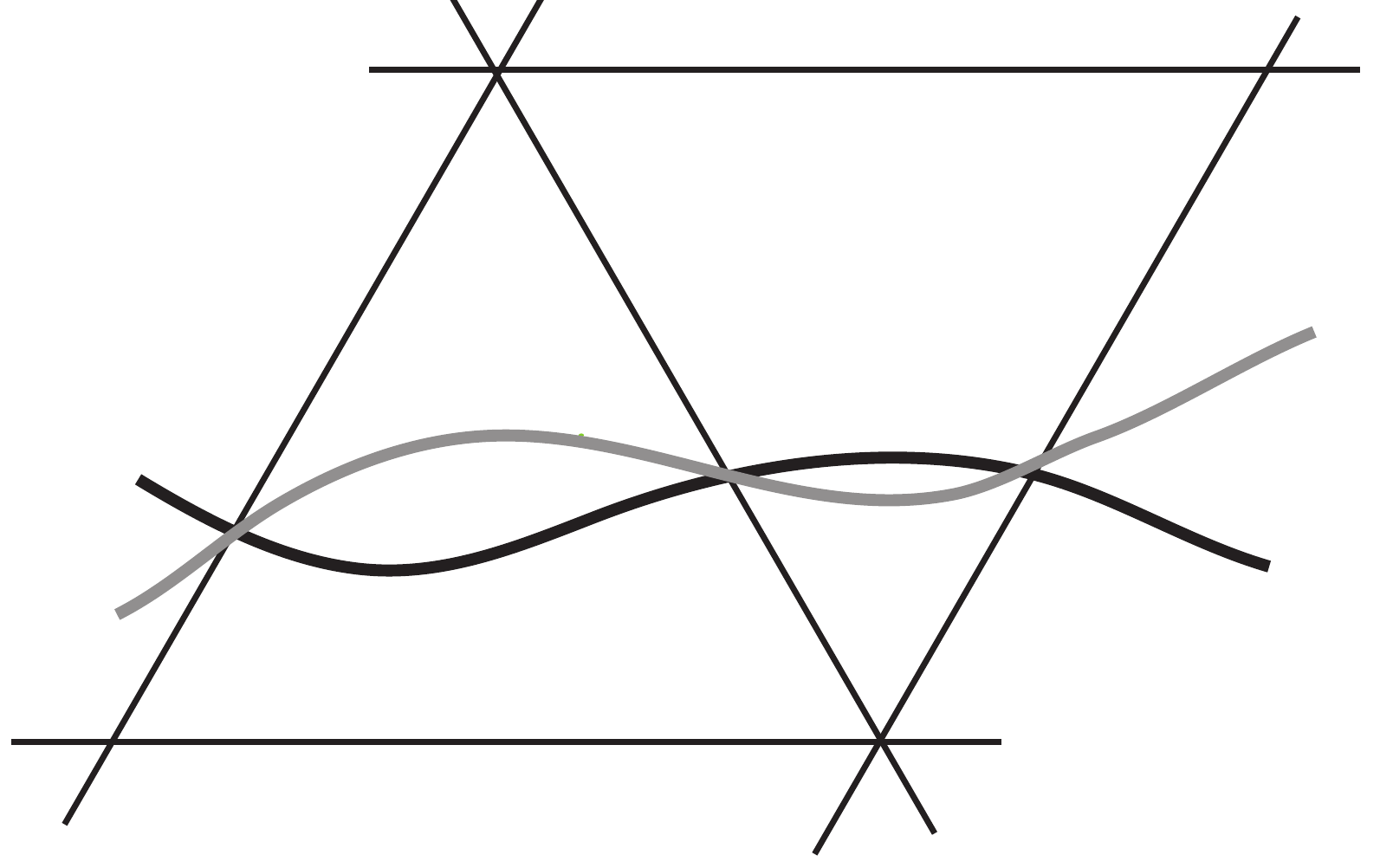}}
		\caption{}
		\label{fig:composewithtemplate1} 
	\end{subfigure}
	\begin{subfigure}[b]{0.3\textwidth}
		\centering
		\scalebox{.3}{\includegraphics{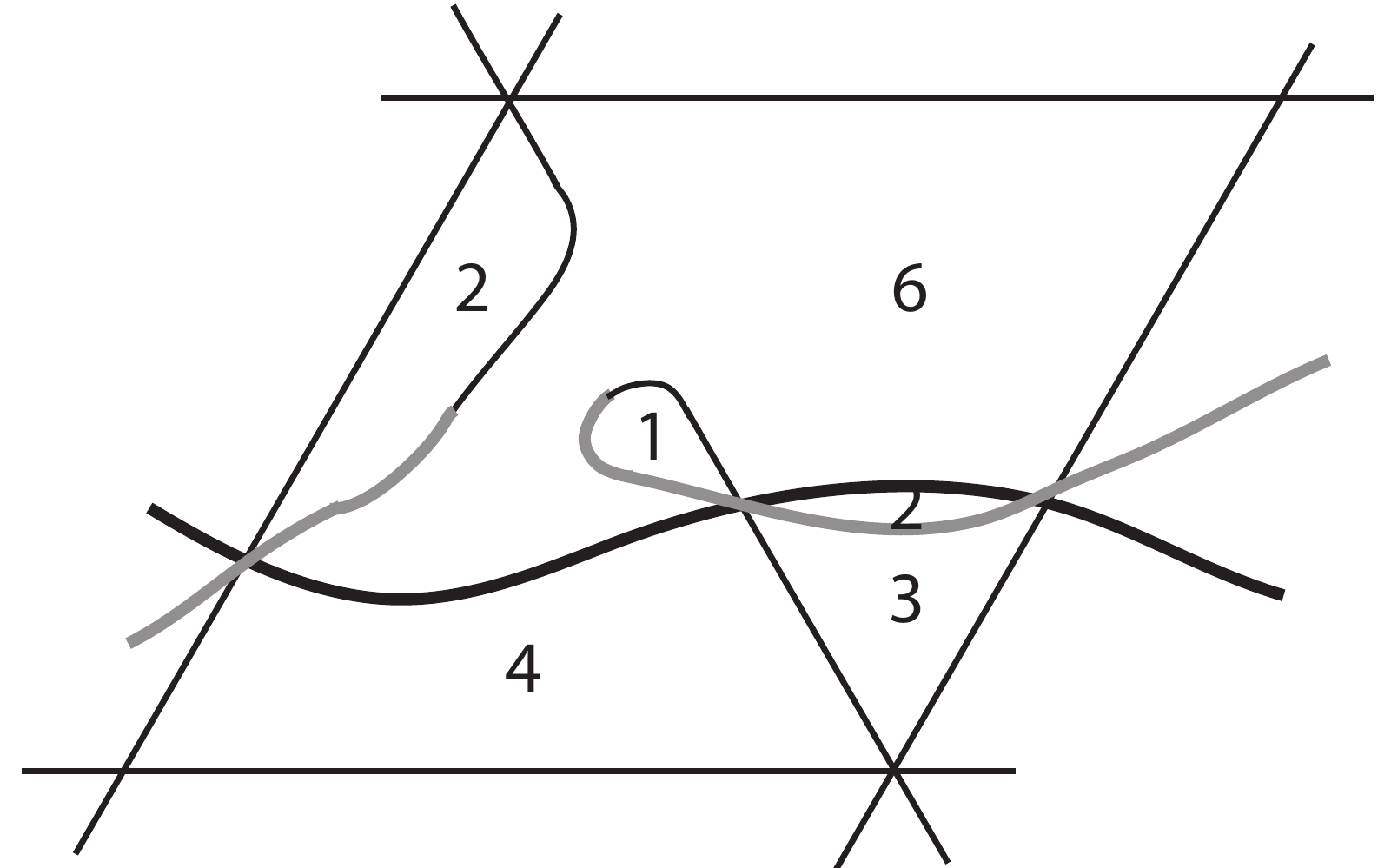}}
		\caption{}
		\label{fig:composewithtemplate2} 
	\end{subfigure}
	\begin{subfigure}[b]{0.3\textwidth}
		\centering
		\scalebox{.3}{\includegraphics{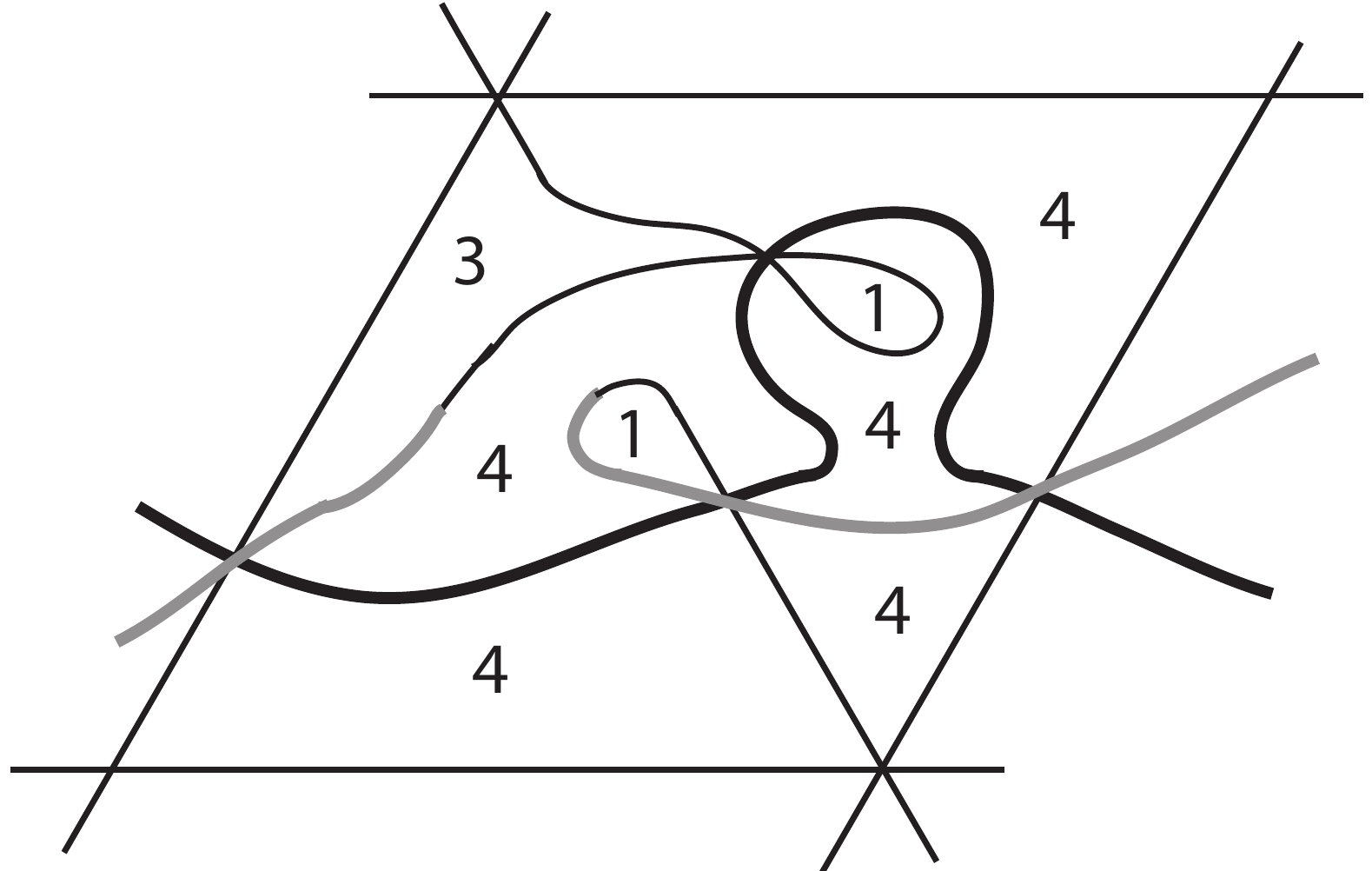}}
		\caption{}
		\label{fig:composewithtemplate3} 
	\end{subfigure}
\caption{Composition with the template knot.  (a) Beginning projection.  (b) Composition of the doubling knot and the template.  
		(c) Trivial 3-crossing to fix the 6-gon.}
\label{fig:compositionwithtemplate} 
\end{figure}

Notice that we need at least two consecutive triangles as in Figure \ref{fig:trianglewithsinglestranddoubled}.  
However, this is guaranteed by the requirement 
that the square grid be fine enough to contain two consecutive squares in a row that contain only single horizontal strands.
\end{proof}


\section{(3n)-Crossing Projections}

This section will prove that all $(3n)$-crossing projections realize the sequence $(1,2,3,4)$, for all $n\geq 1$.  
This proof relies on the following method, which we will call the \emph{doubling method}.  
This method is a generalization of the doubling used in the above proof.  Given an $n$-crossing knot projection $P$, 
take a copy $P'$ of $P$ with crossing data changed so that $P'$ is trivial.  Lay $P'$ directly on top of $P$, and perturb $P'$ slightly so that $P'$ touches $P$ only 
at crossing points, and $P'$ travels straight through the crossing. If there are an even number of crossings in $P$ or if $n$ is even, this 
yields a $(2n)$-crossing link projection of $P$ and a trivial knot.  If there are an odd number of crossings in $P$ and $n$ is odd, this yields a $(2n)$-crossing 
projection of $P$, since this will cause $P'$ and $P$ to be composed.

We can use this method to obtain a $(3n)$-crossing template knot, as shown in Figure \ref{fig:3ntemplate}.

\begin{lemma} \label{lem:3ntemplate} There exists a $(3n)$-crossing knot projection with arbitrarily large 
central hexagonal region tiled by equilateral triangles bounded by bigons, which realizes the sequence $(1,2,3)$.
\end{lemma}

\begin{proof}  Take a template knot. Notice that if there are $m$ triangles along a side of the hexagon, then there are $(m+1) + (m+2) + \dots + (m+m+1) + 
(m+m) + (m+m-1) + \dots + (m+1) = 2m^2\sum_{i=1}^{m} i + 2m+1$ crossings, which is an odd number of crossings.  Using the doubling method, we obtain 
a $6$-crossing projection of the trivial knot.  Iterate this process by applying the doubling method to the original template knot so that each doubling knot 
does not touch the other doubling knots except at crossing points of the original template knot.  
This gives a $9$-crossing projection of the trivial knot.  Repeating this yields the desired $(3n)$-crossing knot projection for any value of $n$.
\end{proof}

\begin{figure}[h!]
	\centering
	\scalebox{.8}{\includegraphics{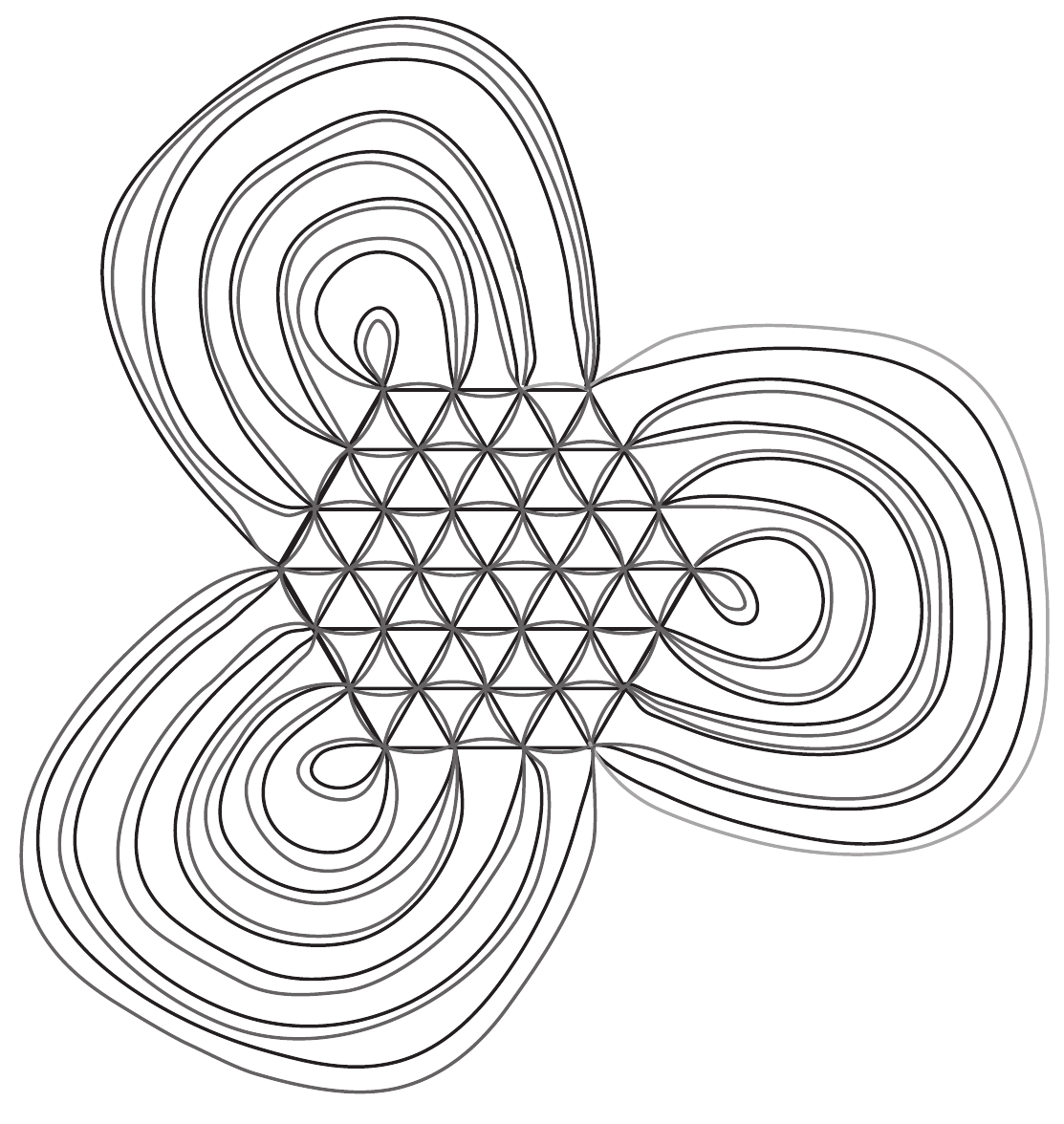}}
	\caption{A $6$-crossing template knot.  The light gray segments in the upper right corner of the hexagon show the point of composition.}
	\label{fig:3ntemplate} 
\end{figure}

We will call a knot projection of this type, with crossing data selected to make it the trivial knot, a \emph{$(3n)$-crossing template knot}. We will also use the move 
displayed in Figure \ref{fig:looptrick}, which we call the \emph{loop trick}.  
This takes two strands and introduces a trivial $n$-crossings between the two strands.  
The loop trick will be used to replace $5$-gons with faces with less numbers of sides.

\begin{figure}[h!]
	\begin{subfigure}[b]{0.3\textwidth}
		\centering
		\scalebox{.3}{\includegraphics{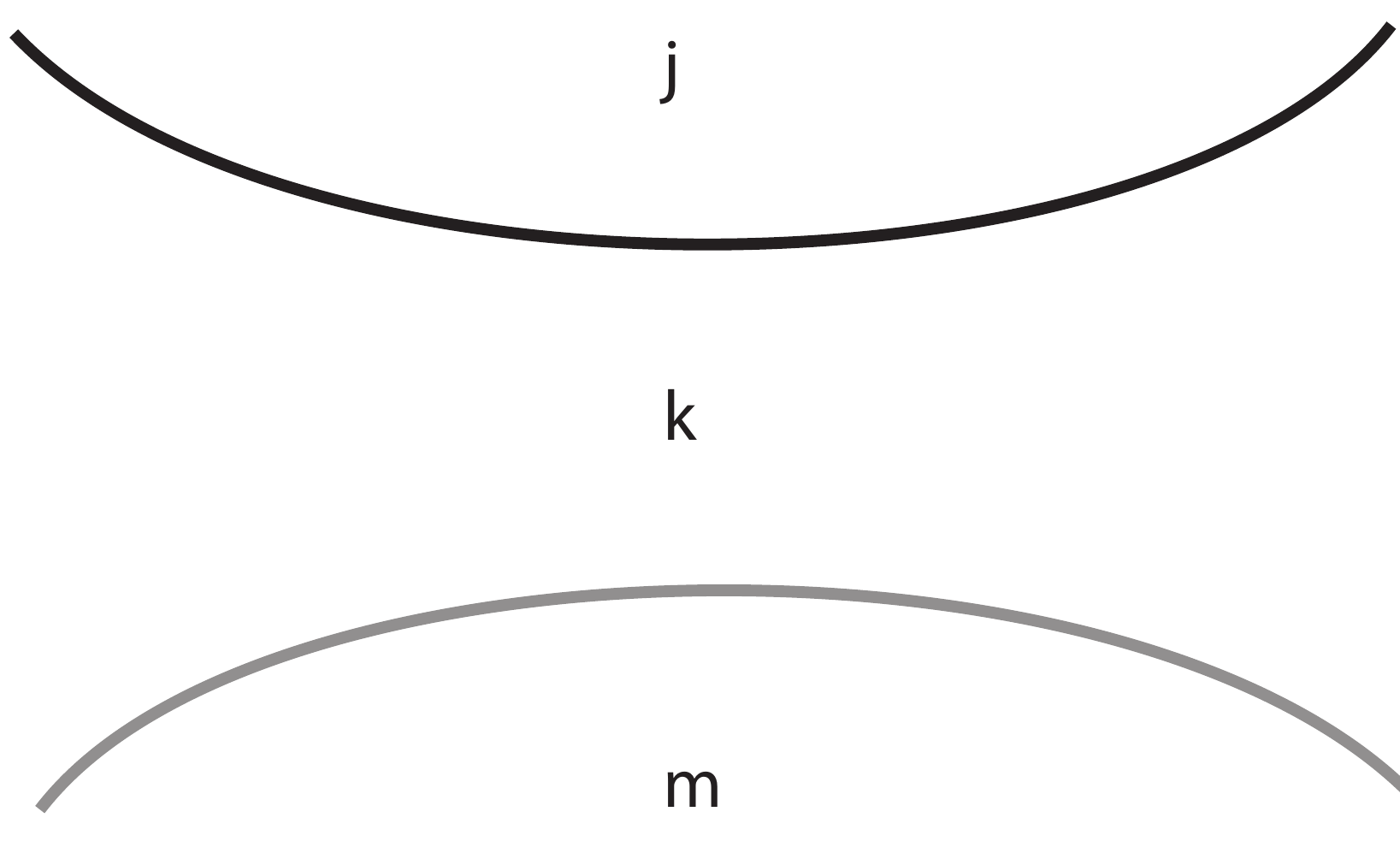}}
		\caption{}
		\label{fig:looptrick1} 
	\end{subfigure}
	\begin{subfigure}[b]{0.3\textwidth}
		\centering
		\scalebox{.3}{\includegraphics{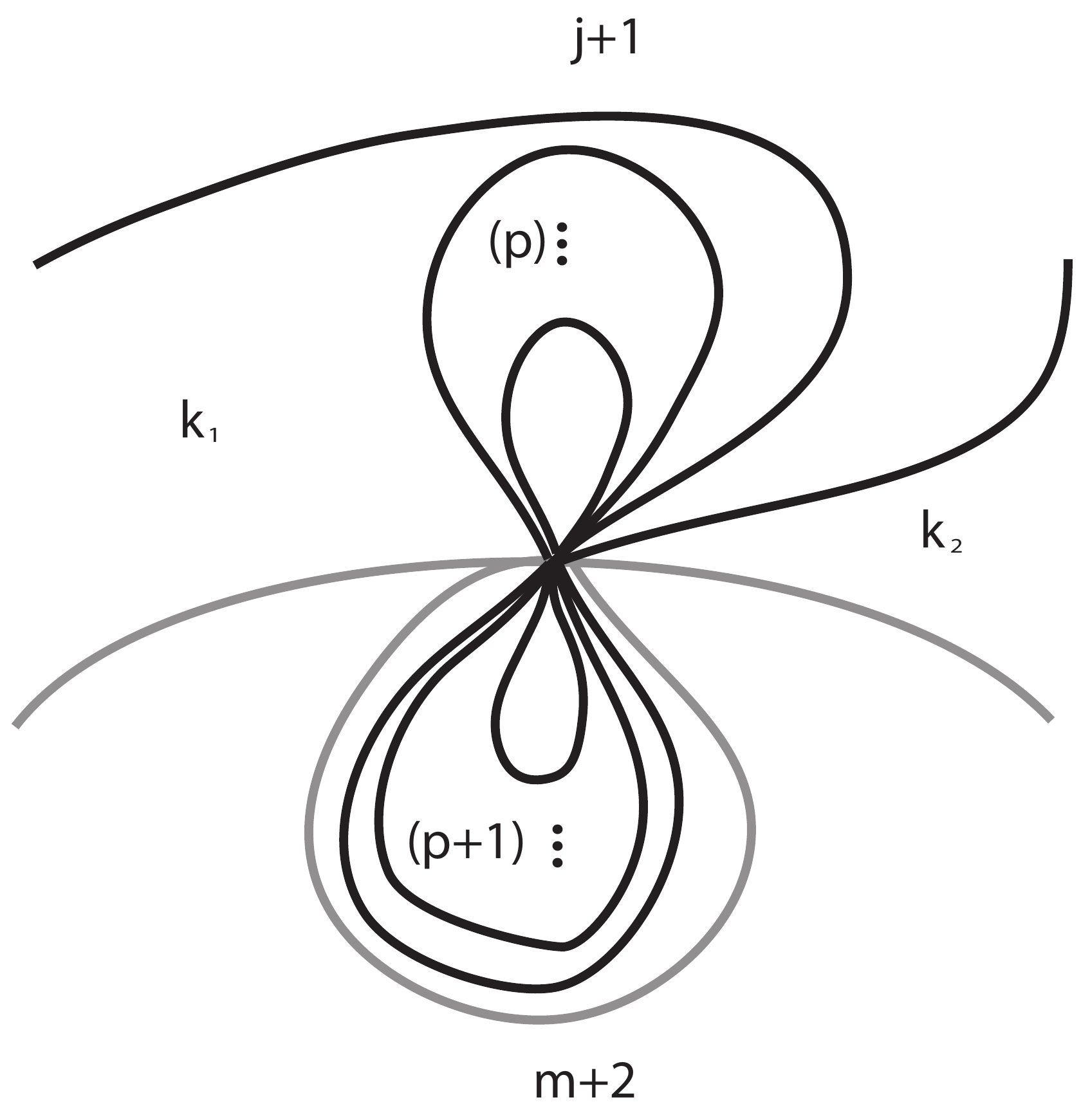}}
		\caption{}
		\label{fig:looptrick2} 
	\end{subfigure}
	\begin{subfigure}[b]{0.3\textwidth}
		\centering
		\scalebox{.3}{\includegraphics{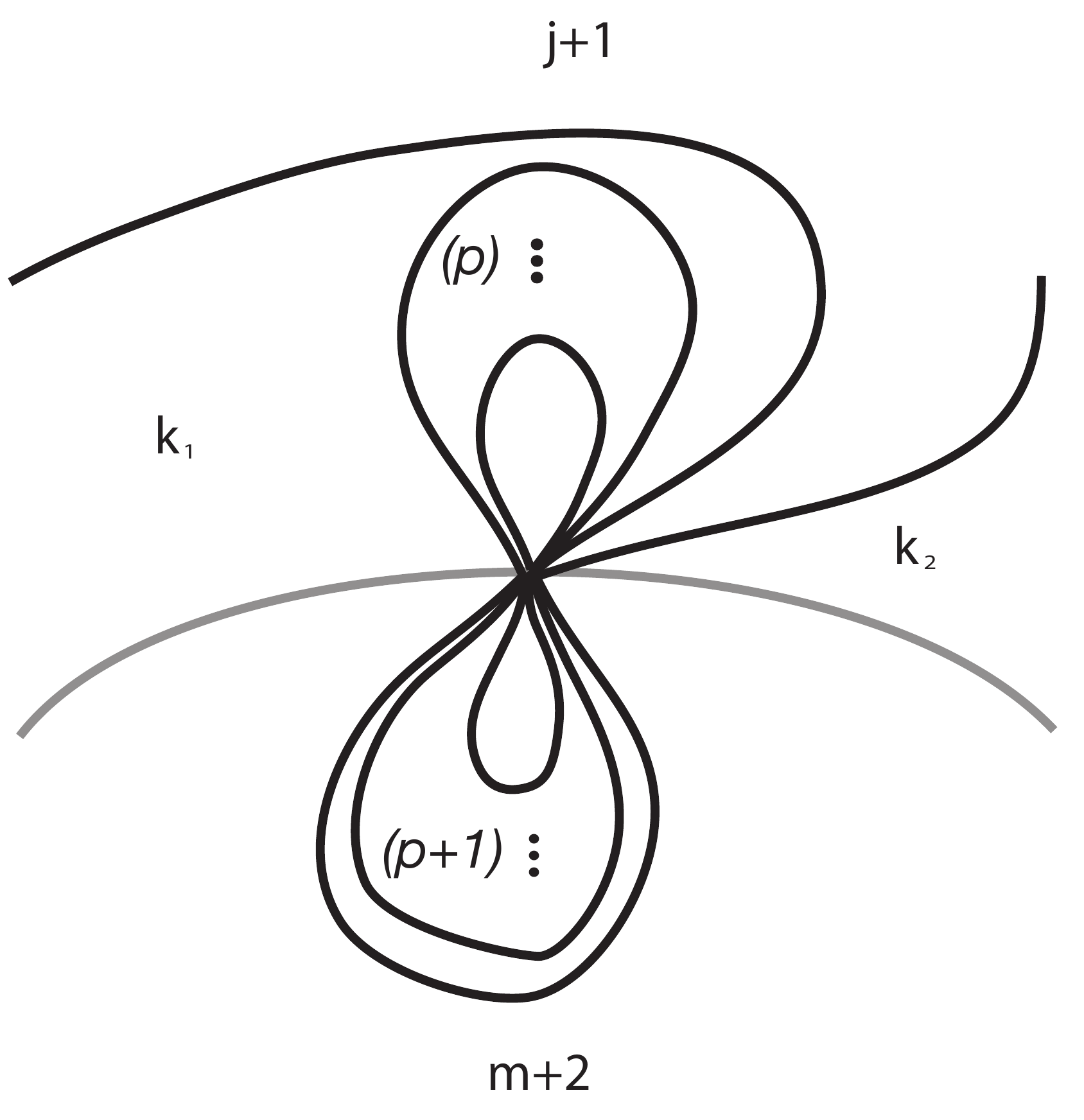}}
		\caption{}
		\label{fig:looptrick3} 
	\end{subfigure}
\caption{The loop trick. (a) Before the loop trick has been applied.  (b) An even multiplicity multi-crossing.  (c) An odd multiplicity multi-crossing.  
The labels $j, k, m$ denote the number of sides 
on each face of the complementary regions.  In (b) and (c), we have $k_1 + k_2 = k + 3$.}
\label{fig:looptrick} 
\end{figure}

\begin{thm} \label{thm:3ncrossing} The sequence $(1,2,3,4)$ is universal for $(3n)$-crossing projections.
\end{thm}

\begin{proof} Given a knot $K$.  Begin with a $(3n)$-crossing template knot $T$ as constructed in Lemma \ref{lem:3ntemplate}.  Lay a rectilinear projection $P$ of $K$ 
on the template as in the $3$-crossing case.  Double this knot.  Take a trivial copy of $P$, named $P'$, and lay it on top of $P$.  Perturb $P'$ slightly so that 
it crosses $P$ only at intersections of $P$ and $T$. The only cases in which $P$ crosses itself are in triangles of Type \ref{fig:trianglewithcrossing}.  
Perturb $P'$ as in Figure \ref{fig:trianglewithdoubledcrossing}.  
Notice this adds one to the multiplicity of 
each intersection of $P$ with $T$.  Repeat this doubling until the multiplicity of each such crossing is $(3n)$. 

After this process, there are three types of triangles, as above.  There is one case of a $5$-gon, which can be fixed using the loop trick.

This yields a link projection of several trivial doubling knots, the original knot, and the trivial template knot.  Compose $P$ and the doubling knots as shown in 
Figure \ref{fig:compositionwithdouble}, and compose $P$ with $T$ as before, shown in Figure \ref{fig:compositionwithtemplate}.  Again, we can size $P$ and 
$T$ appropriately so that there are enough triangles as in Figure \ref{fig:trianglewithsinglestranddoubled} that this is possible.
\end{proof}


\section{$(3n+2)-$ and $(3n+4)-$Crossing Case}

Consider a $(3n)$-crossing template knot with an odd number of triangles along each edge of the central hexagonal region.  
Add a loop of the type in Figure \ref{fig:leftloops} to each crossing in the top row and alternating rows, and a loop of 
the type in Figure \ref{fig:rightloops} to the crossings in the other rows.  Notice that these loops are 
not the same as the type in the loop trick.  This gives a $(3n+2)$-crossing knot with a central hexagonal region tiled by 
triangles with an extra loop and bigons.  Further, this knot realizes the sequence $(1, 2, 3, 4)$.  Choose crossing data so that this knot is trivial.  
We will call this knot the \emph{$(3n+2)$-crossing template knot}.

\begin{figure}[h!]
	\begin{subfigure}[b]{0.3\textwidth}
		\centering
		\scalebox{.4}{\includegraphics{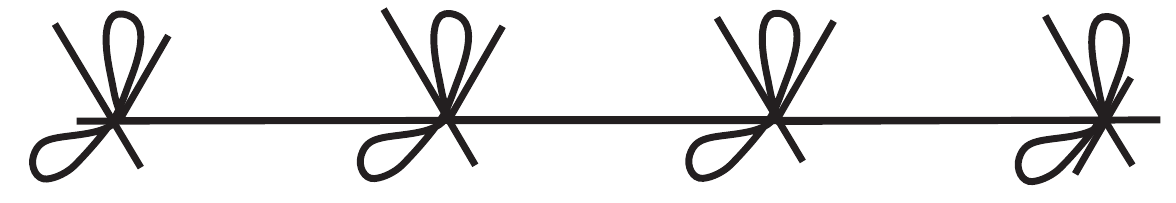}}
		\caption{}
		\label{fig:leftloops} 
	\end{subfigure}
	\begin{subfigure}[b]{0.3\textwidth}
		\centering
		\scalebox{.4}{\includegraphics{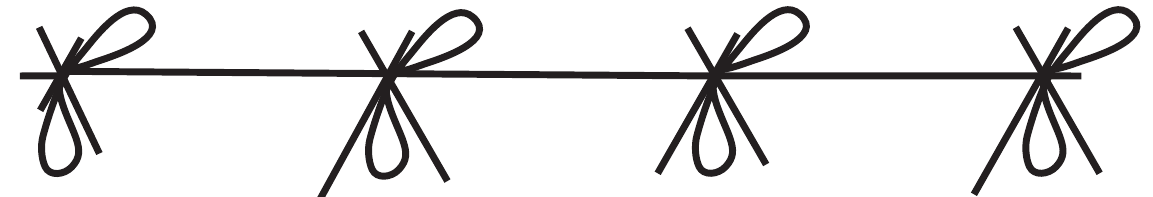}}
		\caption{}
		\label{fig:rightloops} 
	\end{subfigure}
\caption{(a) Type of loops on the top row of template knot and alternating rows below that.  (b) Type of loops on second row and alternating rows below that.}
\label{fig:3n2loops} 
\end{figure}

We will use this template knot to prove that the sequence $(1,2,3,4)$ is universal for $(3n+2)$- and $(3n+4)$-crossing knots.  

\begin{thm}\label{3ncrossing} The sequence $(1,2,3,4)$ is universal for all $(3n+2)$-crossing knot projections.
\end{thm}

\begin{proof} Given a knot $K$, take a polygonal projection $P$ as above and lay it on top of a $(3n+2)$-crossing template knot $T$ so that 
all the horizontal strands of $P$ lie on rows of the kind highlighted in Figure \ref{fig:3n2templatecolored}. Double this 
projection as in the proof of Theorem \ref{thm:3ncrossing} so that all the crossings have multiplicity $3n+2$.  This gives the four types of 
triangles shown in Figure \ref{fig:3n2triangles}.  Each of these triangles yields a $5$-gon, which can be fixed with the loop trick as shown in Figure \ref{fig:3n25gonsfixed}.

\begin{figure}
	\centering
	\scalebox{.5}{\includegraphics{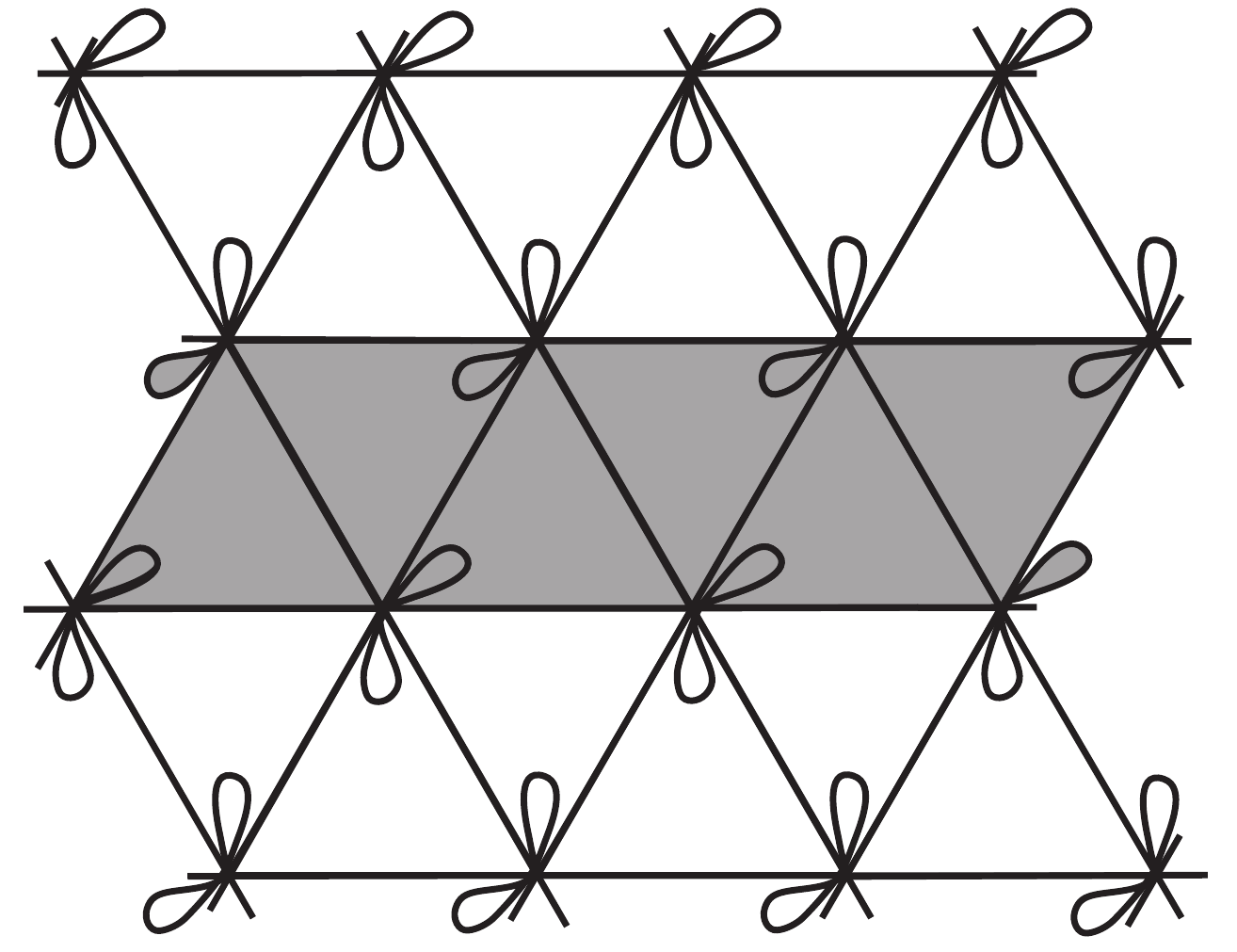}}
	\caption{}
	\label{fig:3n2templatecolored}
\end{figure}

\begin{figure}[h!]
	\begin{subfigure}[b]{0.3\textwidth}
		\centering
		\scalebox{.3}{\includegraphics{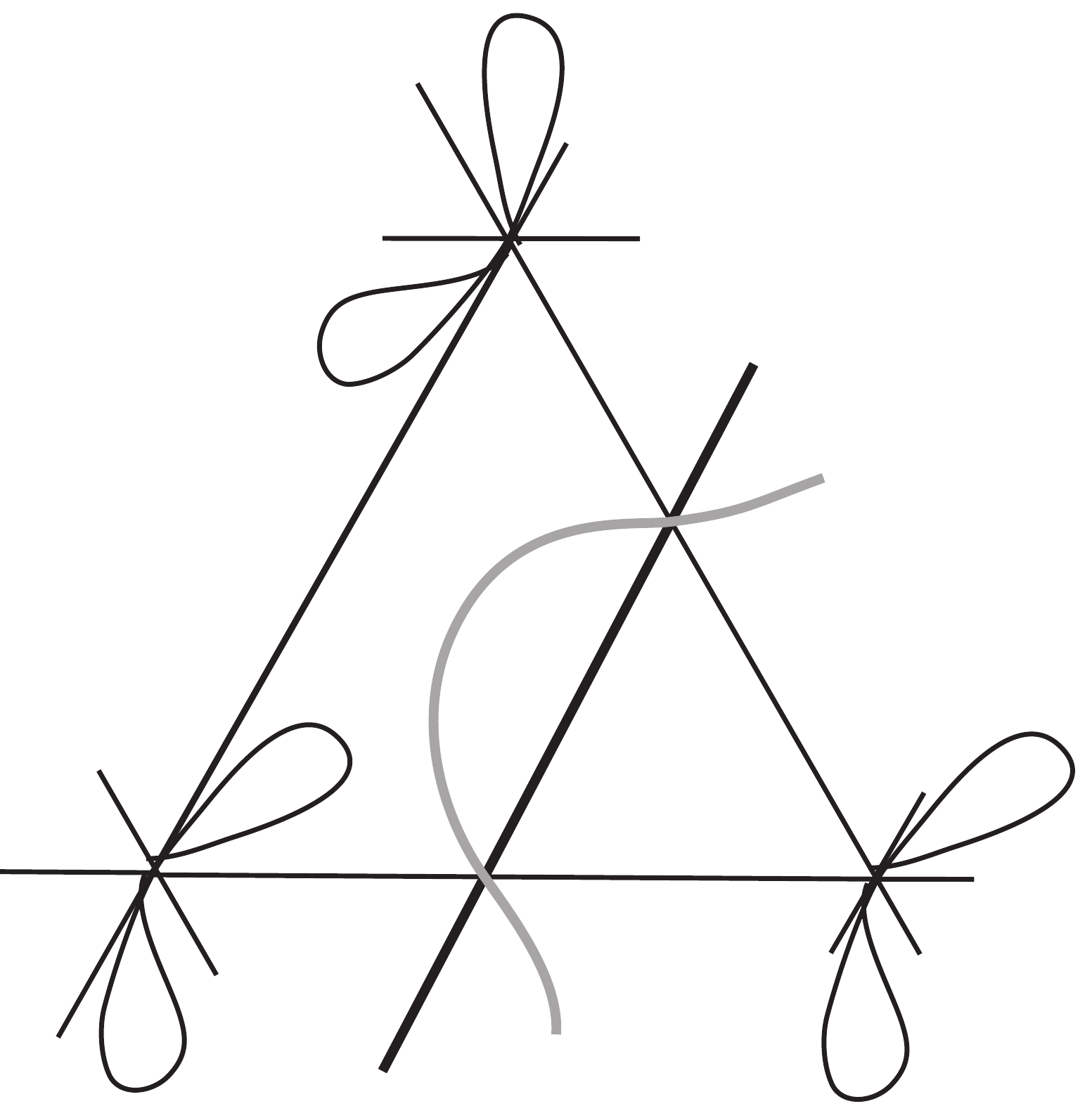}}
		\caption{}
		\label{fig:3n2singlestrand} 
	\end{subfigure}
	\begin{subfigure}[b]{0.3\textwidth}
		\centering
		\scalebox{.3}{\includegraphics{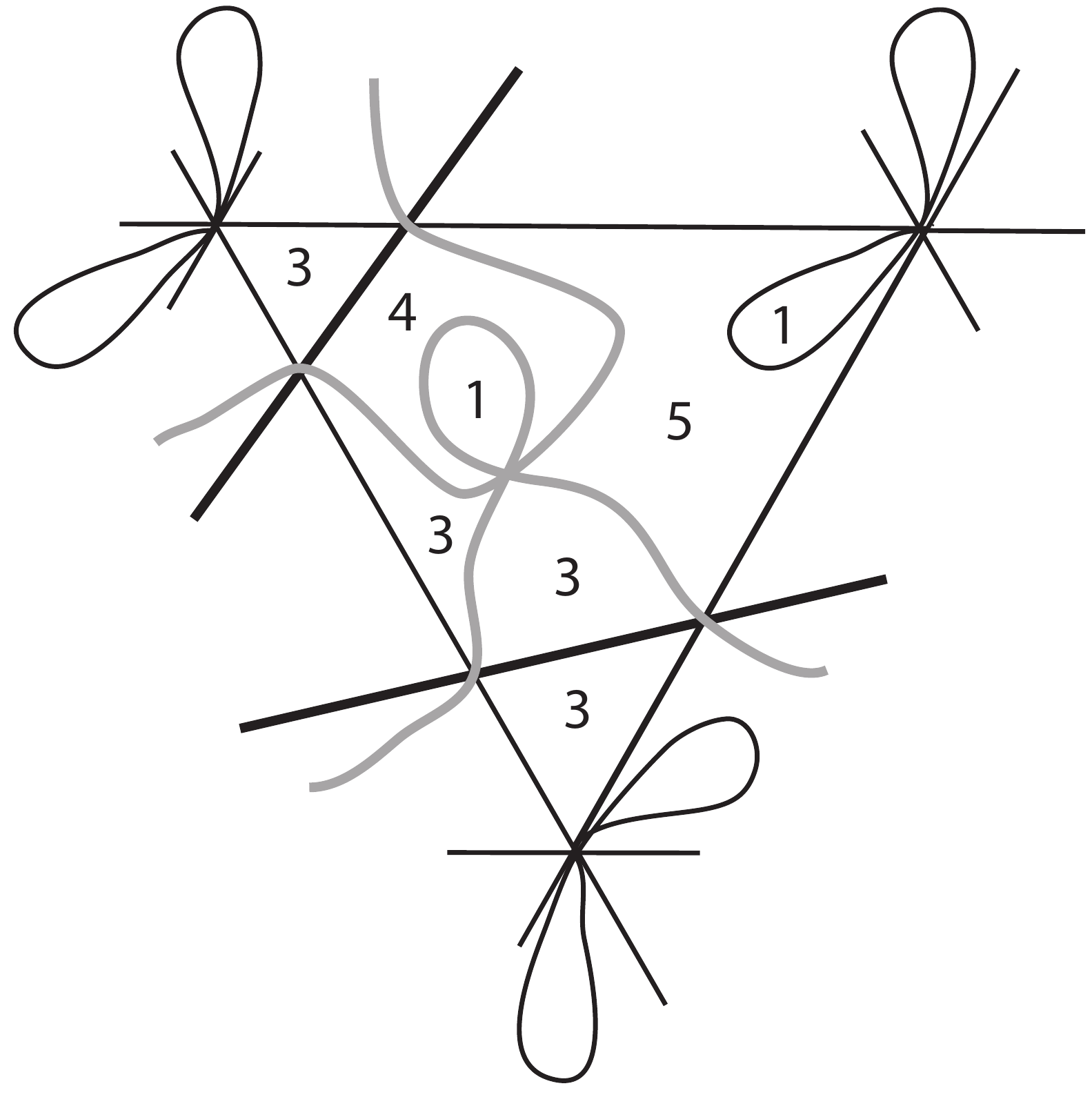}}
		\caption{}
		\label{fig:3n2twouncrossingstrands2} 
	\end{subfigure}
	\begin{subfigure}[b]{0.3\textwidth}
		\centering
		\scalebox{.3}{\includegraphics{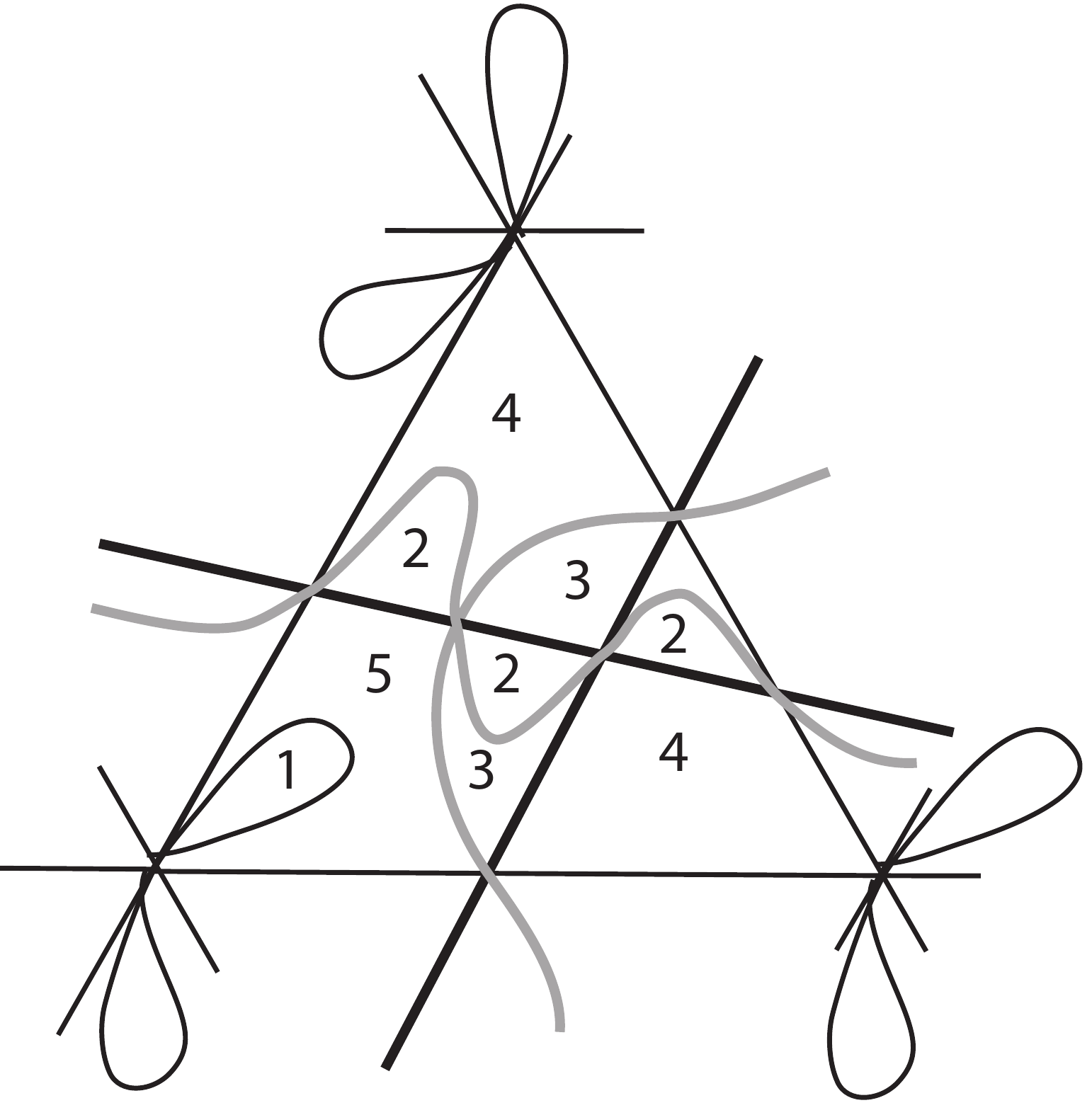}}
		\caption{}
		\label{fig:3n2crossingstrandsA2} 
	\end{subfigure}
	\begin{subfigure}[b]{0.3\textwidth}
		\centering
		\scalebox{.3}{\includegraphics{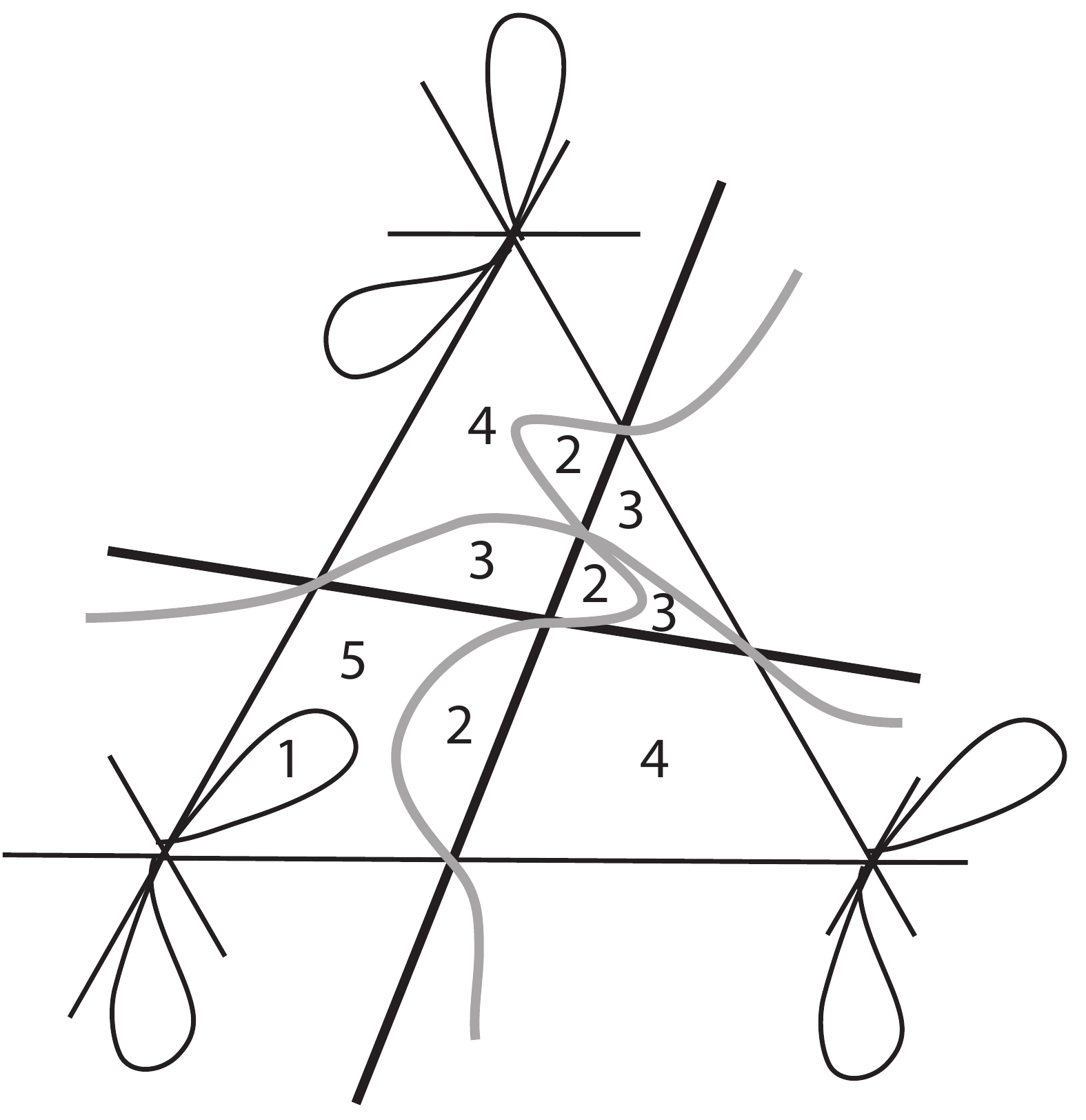}}
		\caption{}
		\label{fig:3n2crossingstrandsB2} 
	\end{subfigure}
\caption{}
\label{fig:3n2triangles} 
\end{figure}

\begin{figure}[h!]
	\begin{subfigure}[b]{0.3\textwidth}
		\centering
		\scalebox{.3}{\includegraphics{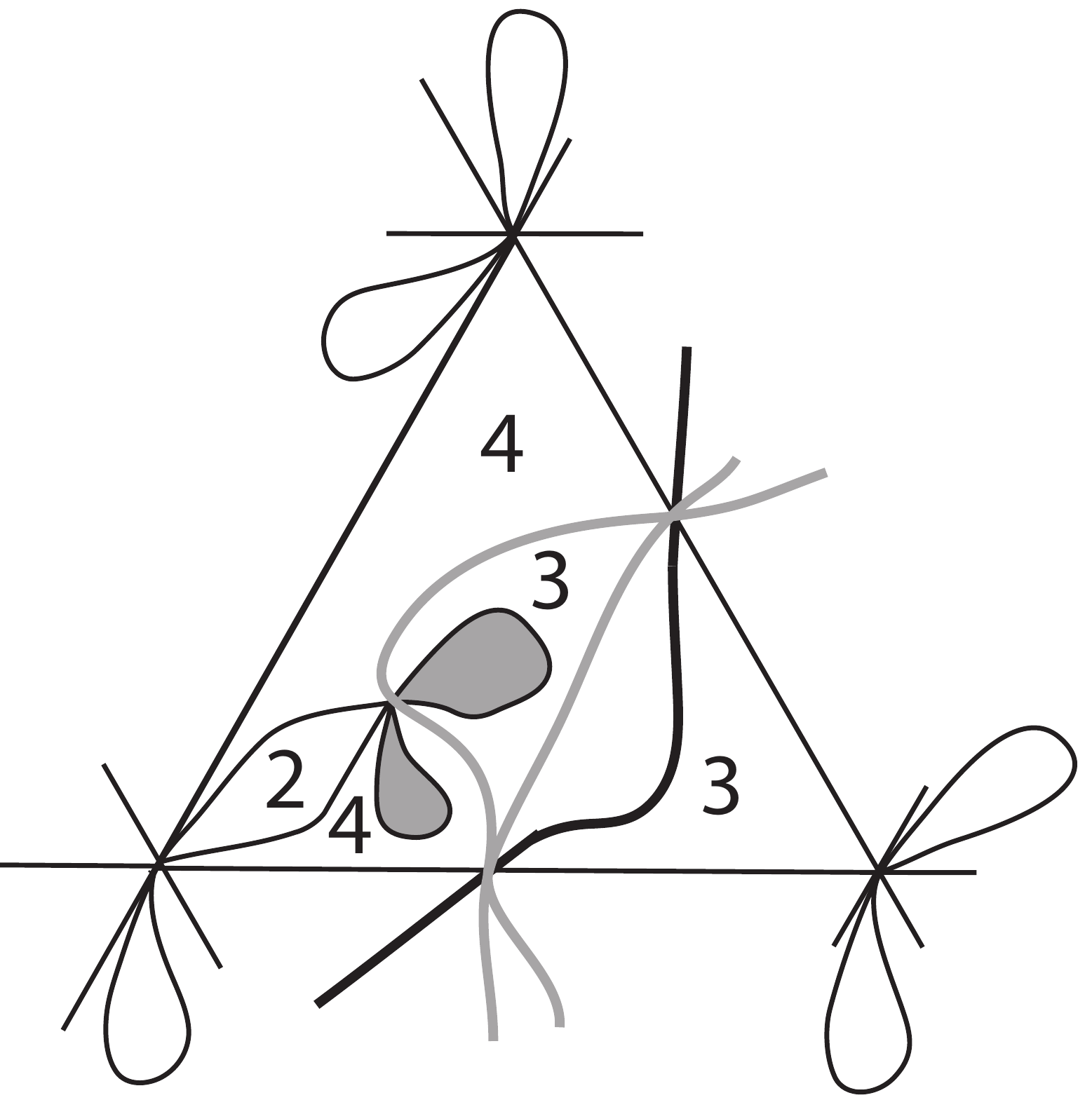}}
		\caption{}
		\label{fig:3n2singlestrand} 
	\end{subfigure}
	\begin{subfigure}[b]{0.3\textwidth}
		\centering
		\scalebox{.3}{\includegraphics{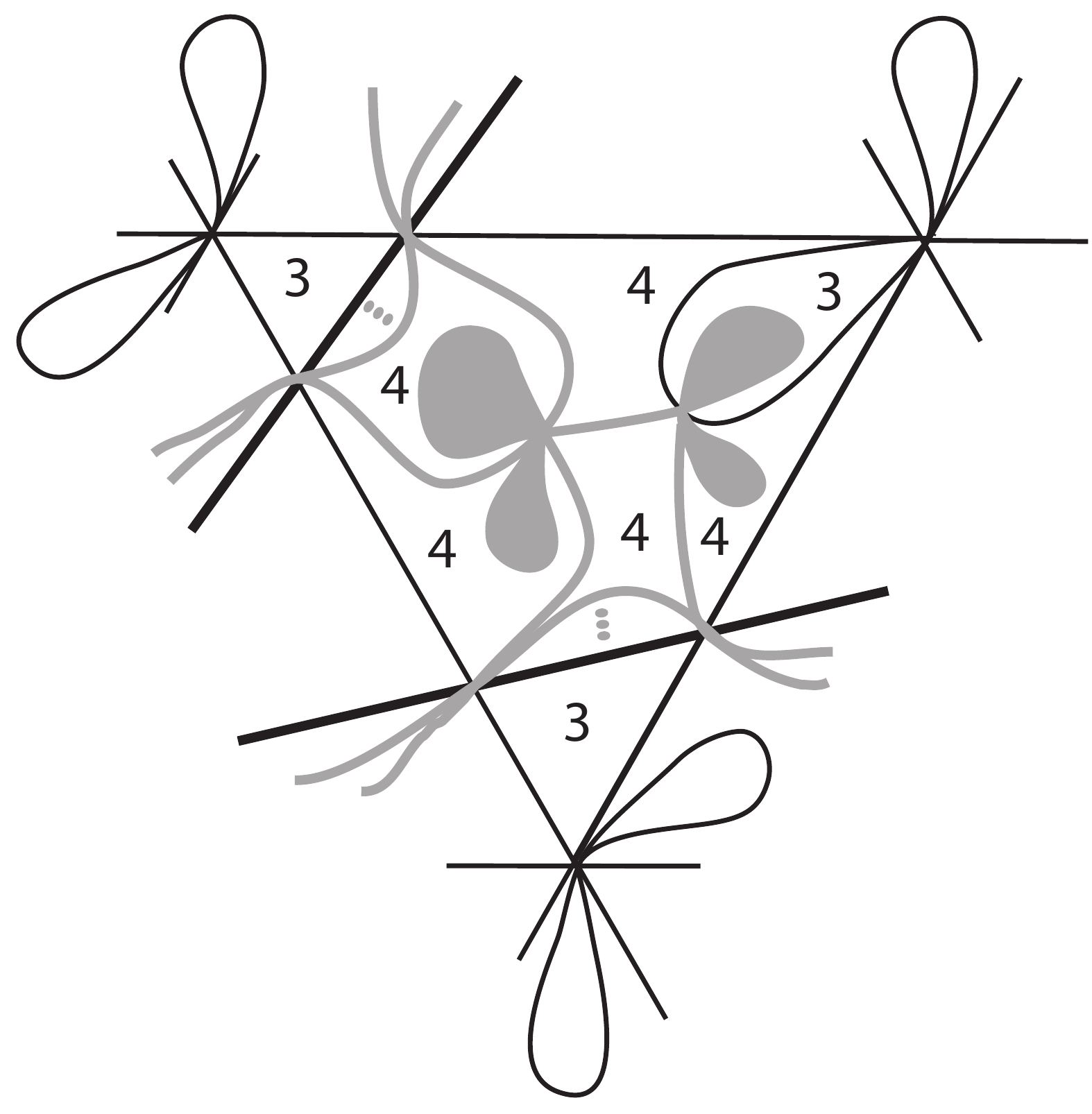}}
		\caption{}
		\label{fig:3n2twouncrossingstrands2} 
	\end{subfigure}
	\begin{subfigure}[b]{0.3\textwidth}
		\centering
		\scalebox{.3}{\includegraphics{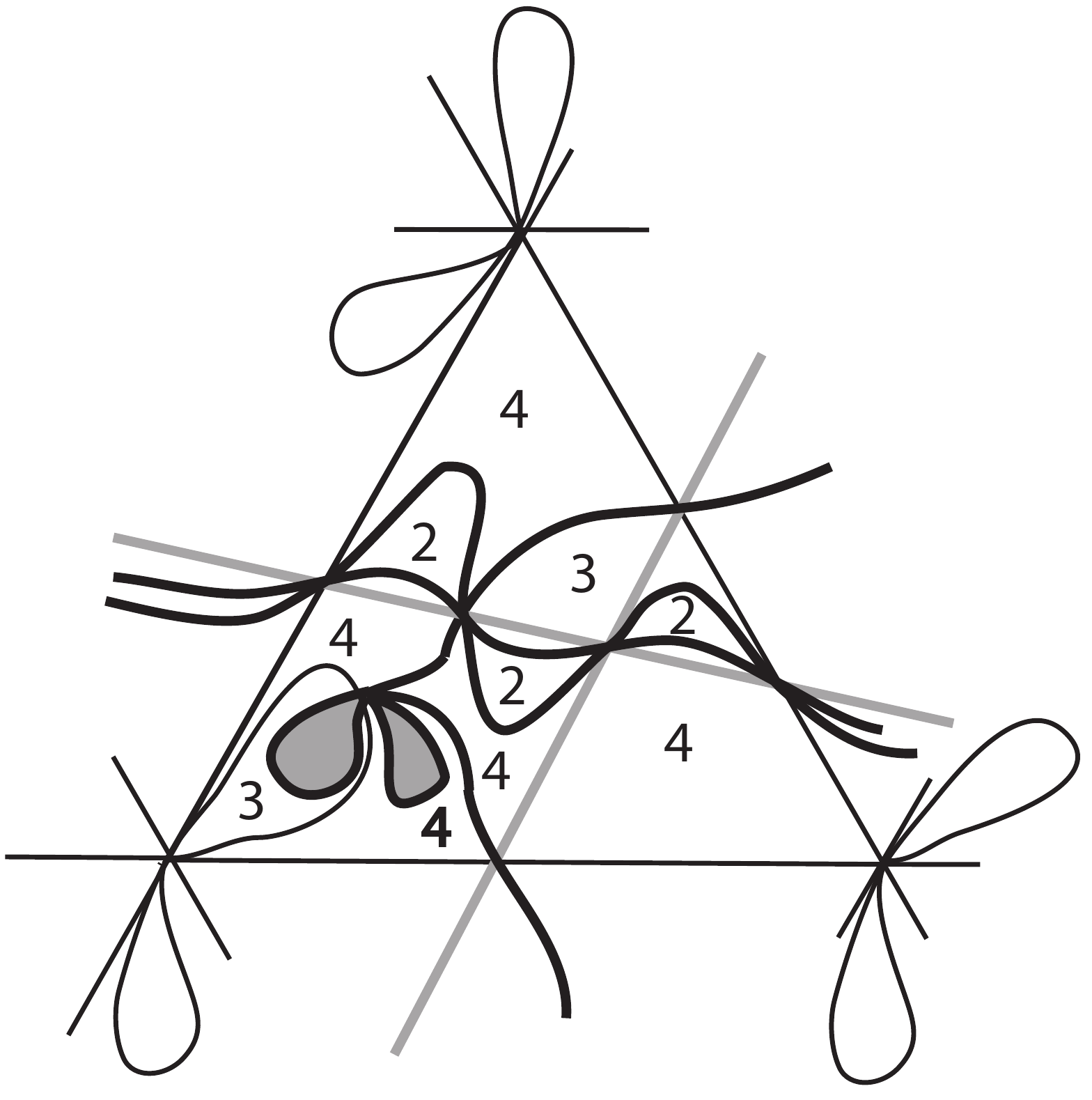}}
		\caption{}
		\label{fig:3n2crossingstrandsA2} 
	\end{subfigure}
	\begin{subfigure}[b]{0.3\textwidth}
		\centering
		\scalebox{.3}{\includegraphics{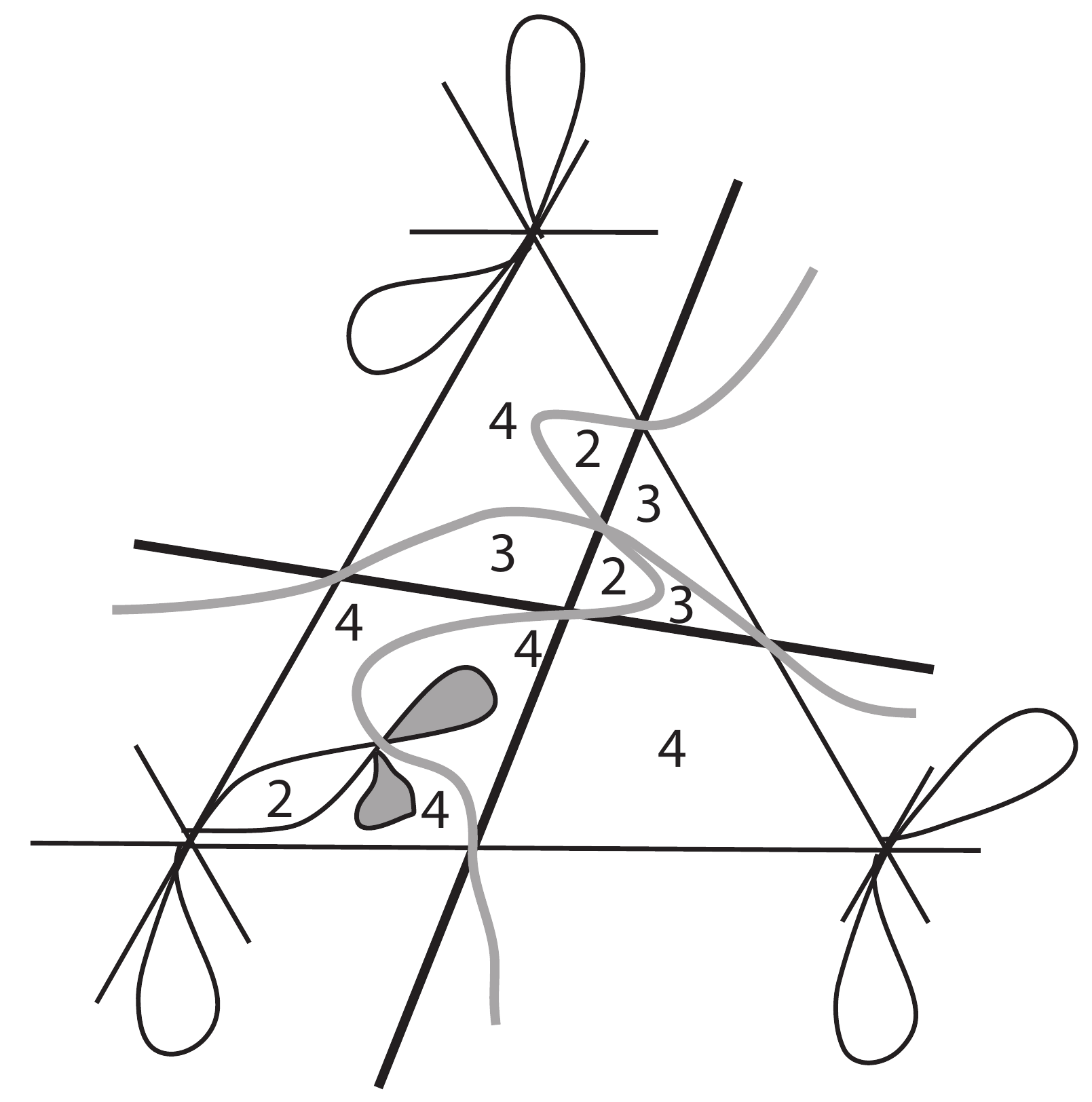}}
		\caption{}
		\label{fig:3n2crossingstrandsB2} 
	\end{subfigure}
\caption{The shaded regions represent nested loops of the kind produced by the loop trick: see Figure \ref{fig:looptrick}}
\label{fig:3n25gonsfixed} 
\end{figure}

We need only compose $P$ with the doubling knots and $P$ with $T$.  We compose as for the $3n$-crossing case.
\end{proof}

We now want to consider the $(3n+4)$-crossing case.  But we can do a very similar trick to show that $(1,2,3,4)$ is universal for this situation; begin with 
a $(3n+2)$-crossing template knot, and add another loop to each of the intersections as in Figure \ref{fig:3n4template}.  This gives a $(3n+4)$-crossing template knot 
which realizes the sequence $(1,2,3,4)$.

\begin{figure}[h!]
	\centering
	\scalebox{.3}{\includegraphics{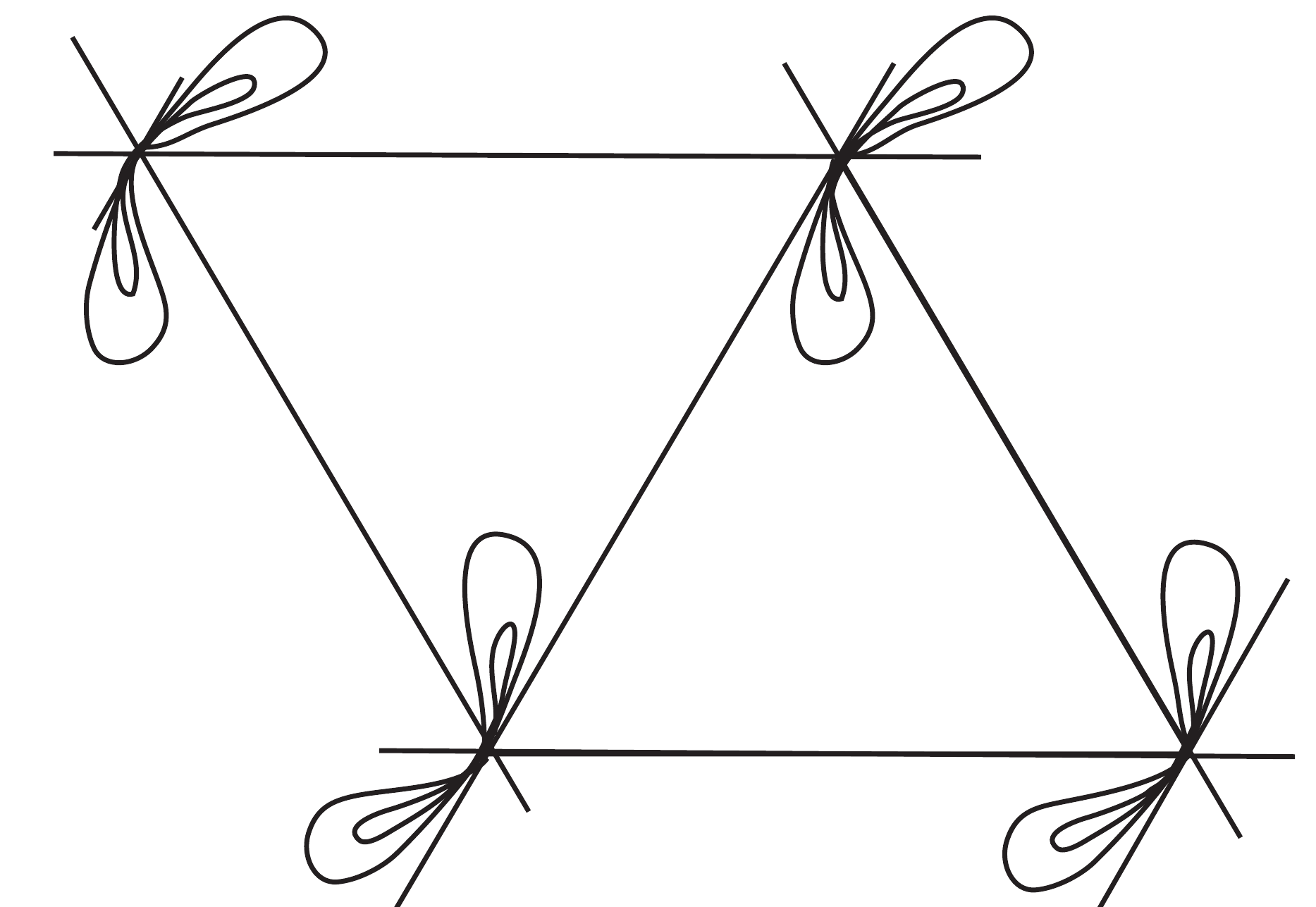}}
\caption{}
\label{fig:3n4template} 
\end{figure}

\begin{thm}  The sequence $(1,2,3,4)$ is universal for all $(3n+2)$-crossing knot projections.
\end{thm}

\begin{proof}  Use the $(3n+4)$-crossing template knot $T$.  Lay a polygonal projection $P$ on top of $T$ as above.  Double $P$ as above with 
enough copies so that each crossing has multiplicity $3n+4$.  This gives the same four types of triangles, which can be fixed the same way. 
Compose as in the $(3n)$-crossing case.
\end{proof}


\section{4-Crossing Case}


We now need only prove that the sequence is universal for $4$-crossing projections.  This will require a different template knot, but the method will be analogous 
to the above proofs.

Consider the knot projection of Figure \ref{fig:4templateex}.  Notice that this is a 4-crossing diagram which 
realizes the sequence $(1, 2, 3, 4)$. Further, it has a central rectangular 
region tiled by squares bounded by bigons.  We want to show that for any $n \times (n+1)$ rectangle where $n$ is even, there exists a knot projection of this form.

\begin{figure}[h!]
	\centering
	\scalebox{.7}{\includegraphics{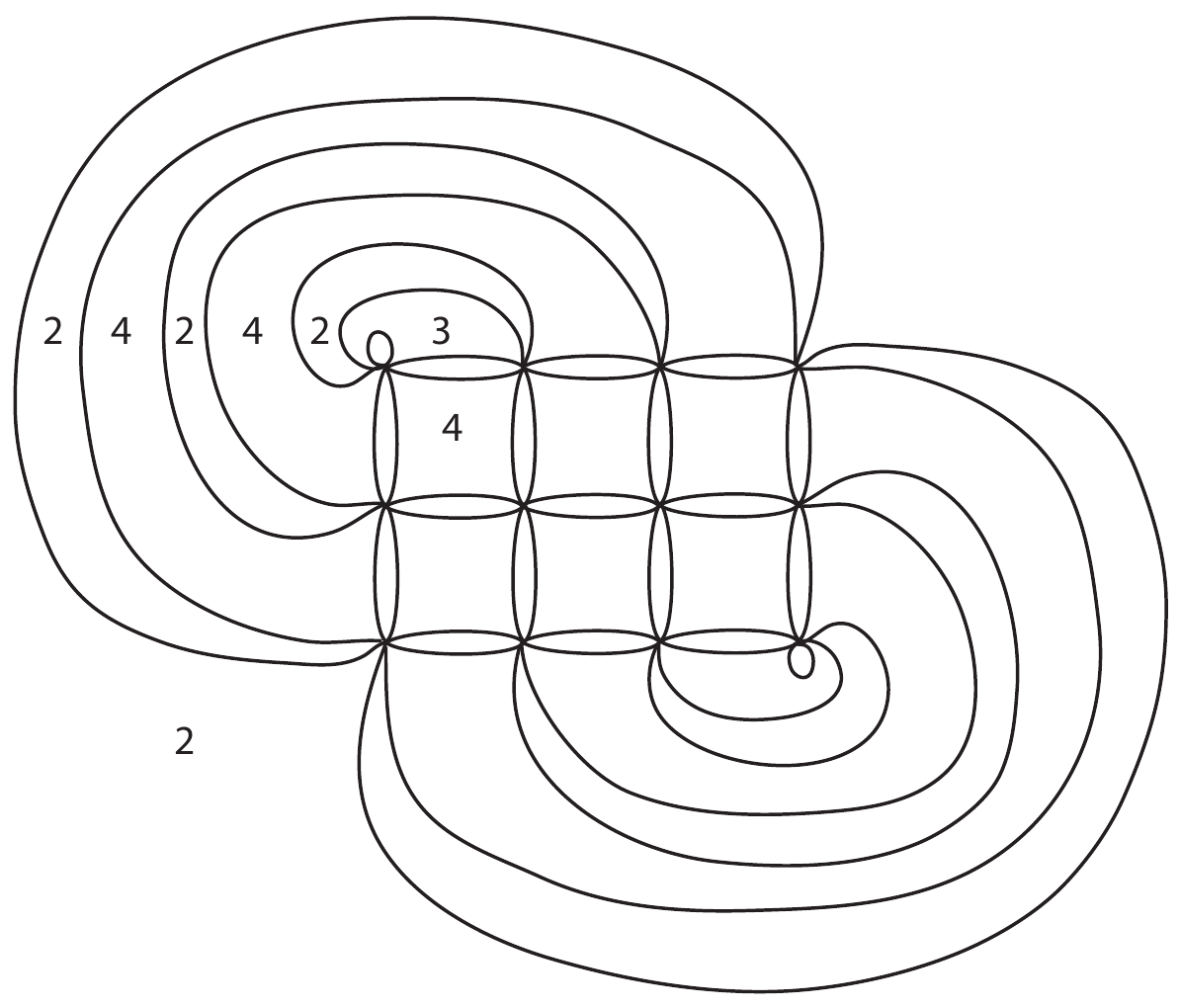}}
	\caption{}
	\label{fig:4templateex}
\end{figure}

\begin{lemma}\label{lemma:4template}  For $n$ an even integer, there exists a 4-crossing knot projection which has an 
$n \times (n+1)$ rectangular central region, tiled as in 
Figure \ref{fig:4templateex}, and which realizes the sequence $(1, 2, 3, 4)$.
\end{lemma}

\begin{proof} Consider an $n \times (n+1)$ grid as in Figure \ref{fig:4generaltemplate}, where $n$ is even.  Since we want each crossing to 
be a quadruple crossing, we have four tail edges at the corners of our grid, and two tail edges at each crossing along the edges of the central 
rectangle.  Notice that there are $4(4) + 2(2(n)+2(n-1)) =  16 + 8n - 4 = 4(2n+3)$ tail edges.  

\begin{figure}
	\centering
	\scalebox{.5}{\includegraphics{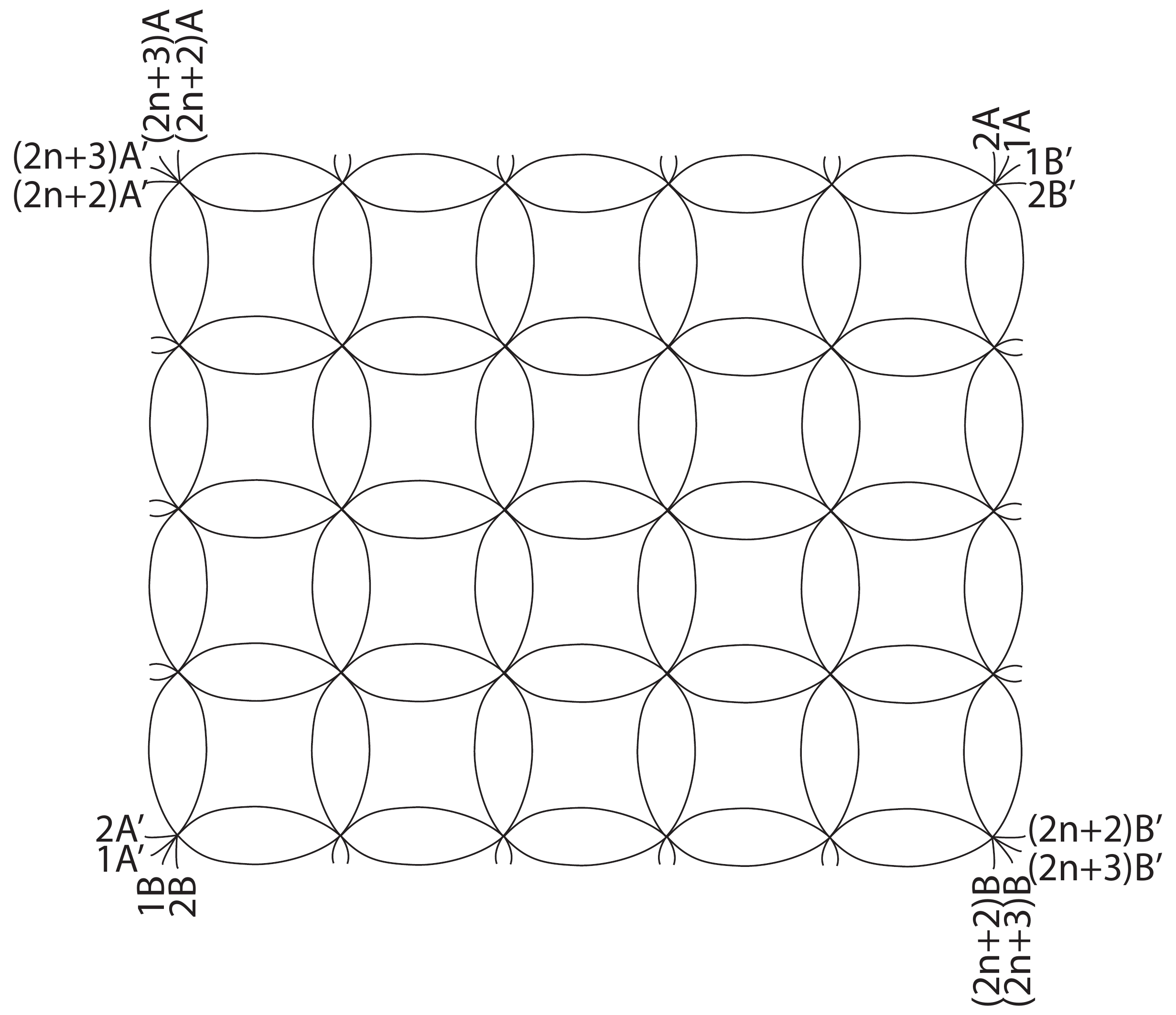}}
	\caption{}
	\label{fig:4generaltemplate}
\end{figure}

Consider the upper right corner.  Label the second most counterclockwise strand $1A$.  Continuing counterclockwise around the 
rectangular region, label the strands $2A, 3A, \dots, (2n+2)A, (2n+3)A, (2n+3)A', (2n+2)A', \dots, 2A', 1A', 1B, 2B, \dots, (2n+2)B, (2n+3)B, 
(2n+3)B', (2n+2)B', \dots, 2B', 1B'$, as in Figure \ref{fig:4generaltemplate}.  For all $i = 1, \dots, 2n+3$, connect the tail edges by the rule
	\begin{align*}
		iA \leftrightarrow iA' \\
		iB \leftrightarrow iB'.	
	\end{align*}

Across the rectangle, we get the following relations.  
	\begin{align*}
		(2i+2)A \leftrightarrow [2(n-i) + 1]B \\
		(2i+2)B \leftrightarrow [2(n-i)+1]A
	\end{align*}
for $i = 1, \dots, (n-1)$.  Notice that this does not include $2A$ or $2B$, since these connect with $B'$ and $A'$ respectively, according to the 
relationship
	\begin{align*}
		2A \leftrightarrow (2n+3)B' \\
		2B \leftrightarrow (2n+3)A'.
	\end{align*}
 For all values of $A'$ and $B'$, we have: 
	\begin{align*}
		(2i+1)A' \leftrightarrow [2(n-i)+2]B' \\
		(2i+1)B' \leftrightarrow [2(n-i)+2]A'
	\end{align*}
for $i = 0, \dots, n$.  Equivalently, we can say that: 
	\begin{align*}
		2iA' \leftrightarrow [2(n-i)+3]B' \\
		2iB' \leftrightarrow [2(n-i)+3] A'
	\end{align*}
for $i = 1, \dots, n+1$.  This gives us every relation between all values of $A, A', B,$ and $B'$.

Define an orientation that travels from $1B$ to $1B'$.  Then following the strand, we get the sequence of tail edges along the knot shown in 
Table \ref{table:4crossingsequence}.  This sequence hits every tail edge of the rectangle.  
Since the interior paths from tail edge to tail edge cover all strands across the rectangle, 
this projection is a knot.
\end{proof}

Pick orderings on the strands so that this knot is the unknot.  We will call an unknot of this type a \emph{$4$-crossing template knot}. 

\begin{thm}\label{thm:4crossing}  The sequence $(1, 2, 3, 4)$ is universal for $4$-crossing knots.
\end{thm}

\begin{proof}  Given a 4-crossing knot $K$, consider a regular projection $P$ of $K$.  Take a polygonal projection $P'$ of $K$ 
such that all the edges of $P'$ are either vertical or horizontal.  Size $P'$ so that given a square grid, all the corner and crossings of $P'$ 
occur at the center of a grid square.  Then we can lay $P'$ on the rectangular grid of our template knot in the same manner as above.

There are three types of squares that can occur in this process, up to rotation: (a) a corner of the original knot, (b) a single strand of the original knot, or 
(c) a crossing of the original knot.  These are depicted in Figure \ref{fig:4knotontemplate}.  

\begin{figure}[!h]
	\begin{subfigure}{.3\textwidth}
		\centering
		\scalebox{.3}{\includegraphics{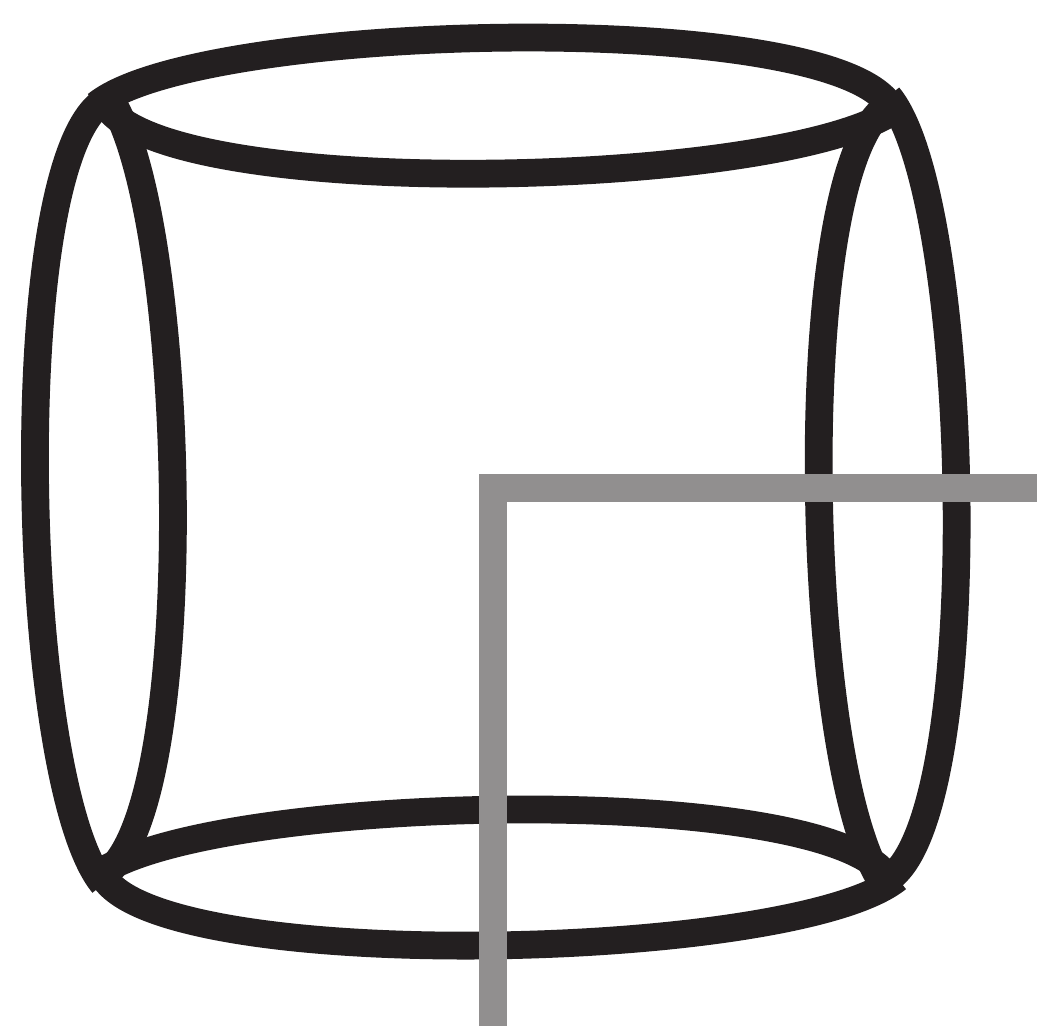}}
		\caption{}
		\label{fig:squarewithcorner}
	\end{subfigure}
	\begin{subfigure}{.3\textwidth}
		\centering
		\scalebox{.3}{\includegraphics{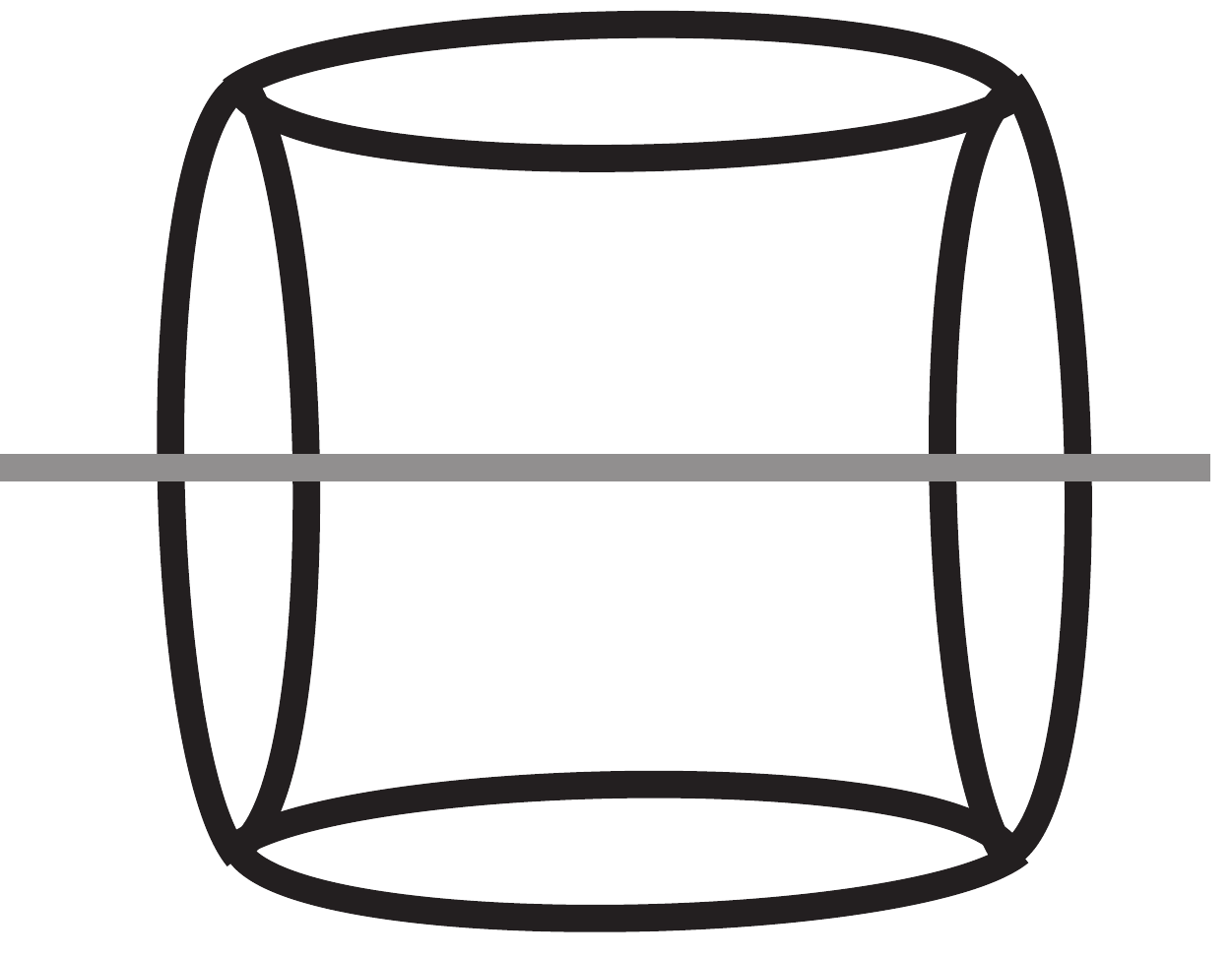}}
		\caption{}
		\label{fig:squarewithcorner}
	\end{subfigure}
	\begin{subfigure}{.3\textwidth}
		\centering
		\scalebox{.3}{\includegraphics{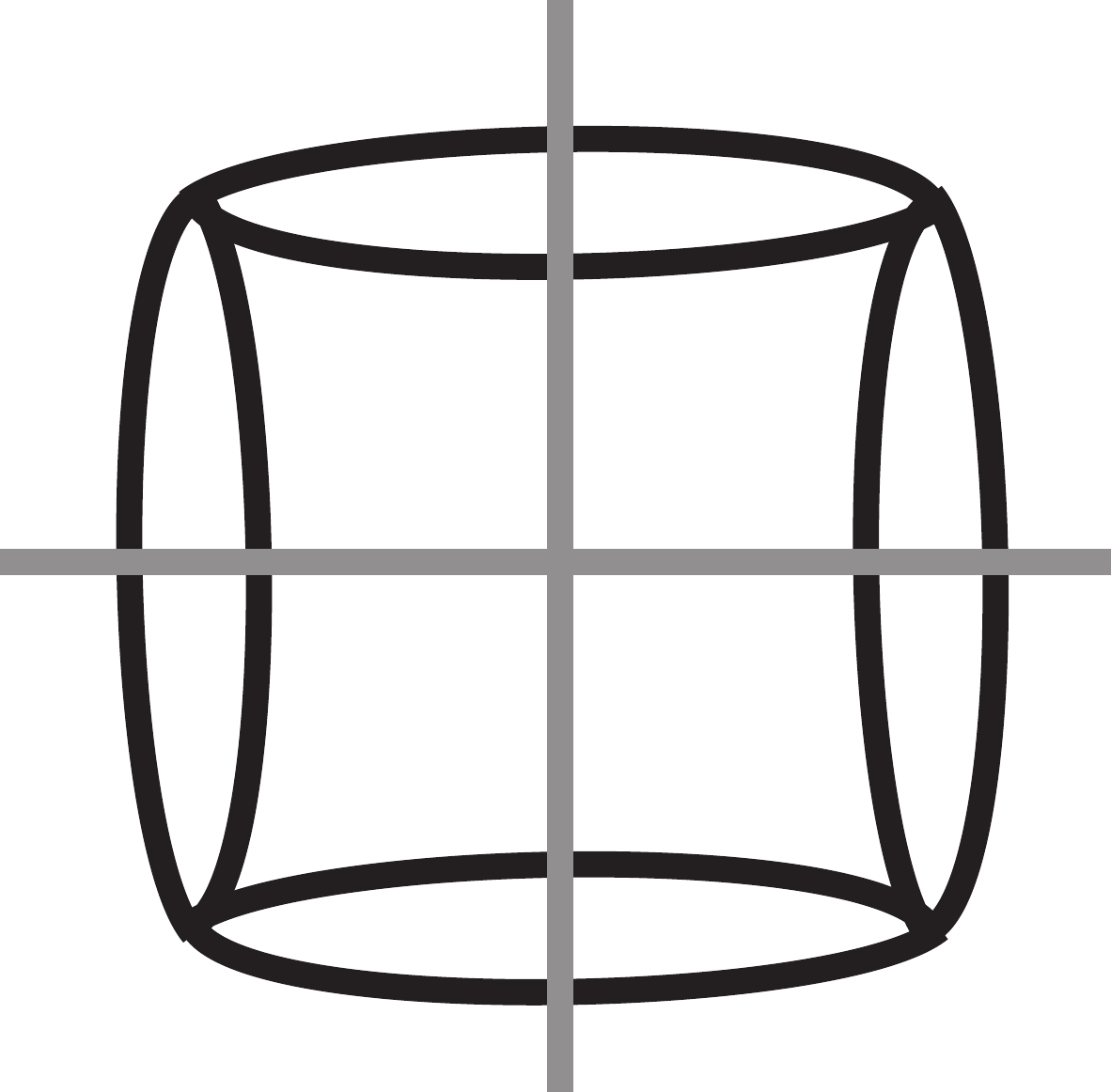}}
		\caption{}
		\label{fig:squarewithcorner}
	\end{subfigure}
	\caption{(a) Corner of the original knot. (b) Single strand of the original knot.  (c) Crossing of the original knot.}
	\label{fig:4knotontemplate}
\end{figure}

Take a trivial copy $P'$ of the original knot.  Lay it on top of the original knot, and perturb it slightly so that it crosses $P$ only at crossing points.  
Double $P$ with a second trivial copy $P''$ in the same manner, so that all crossings have multiplicity four.  Now there are 
three types of squares that can occur: (a) a doubled corner of the original knot, (b) a doubled single strand of the original knot, and (c) a 
doubled crossing of the original knot, as depicted in Figure \ref{fig:4knotontemplatedoubled}.

\begin{figure}[!h]
	\begin{subfigure}{.3\textwidth}
		\centering
		\scalebox{.3}{\includegraphics{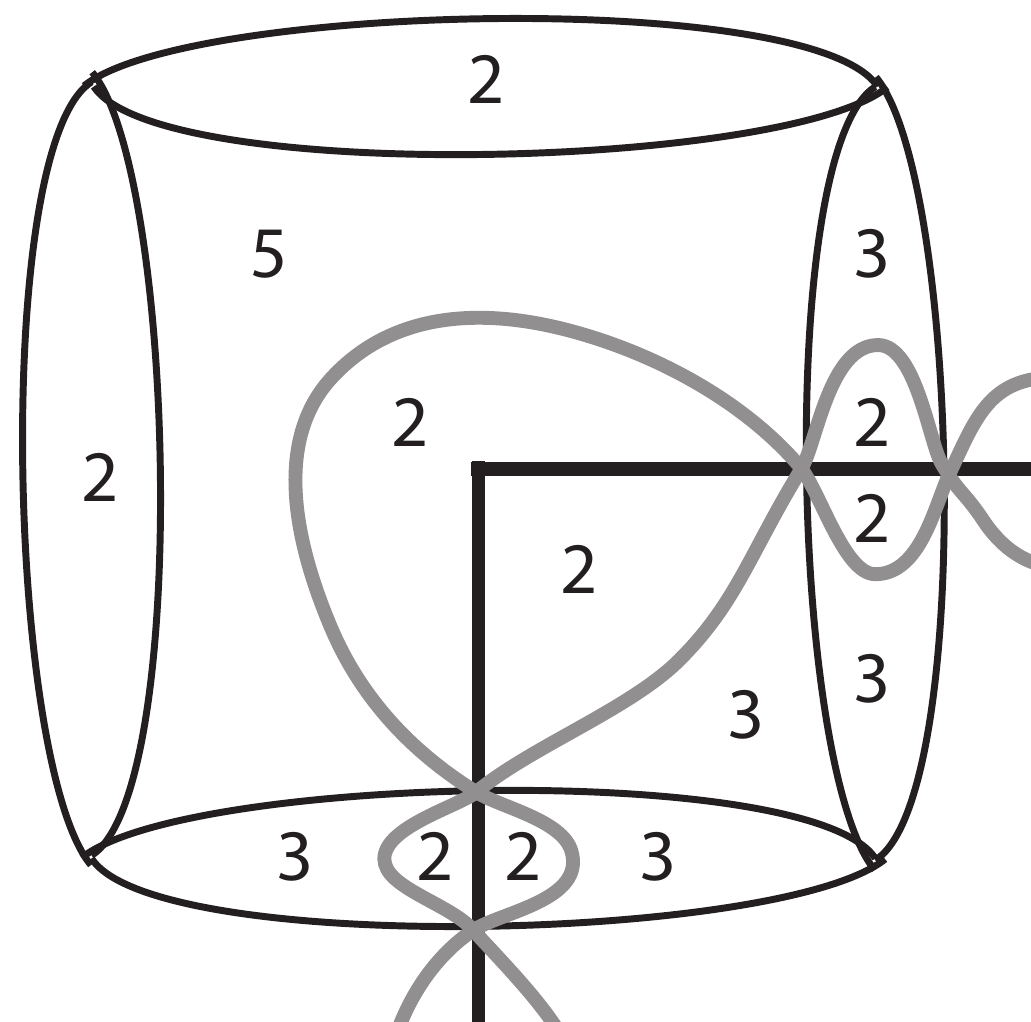}}
		\caption{}
		\label{fig:4squarewithcornerdoubled}
	\end{subfigure}
	\begin{subfigure}{.3\textwidth}
		\centering
		\scalebox{.3}{\includegraphics{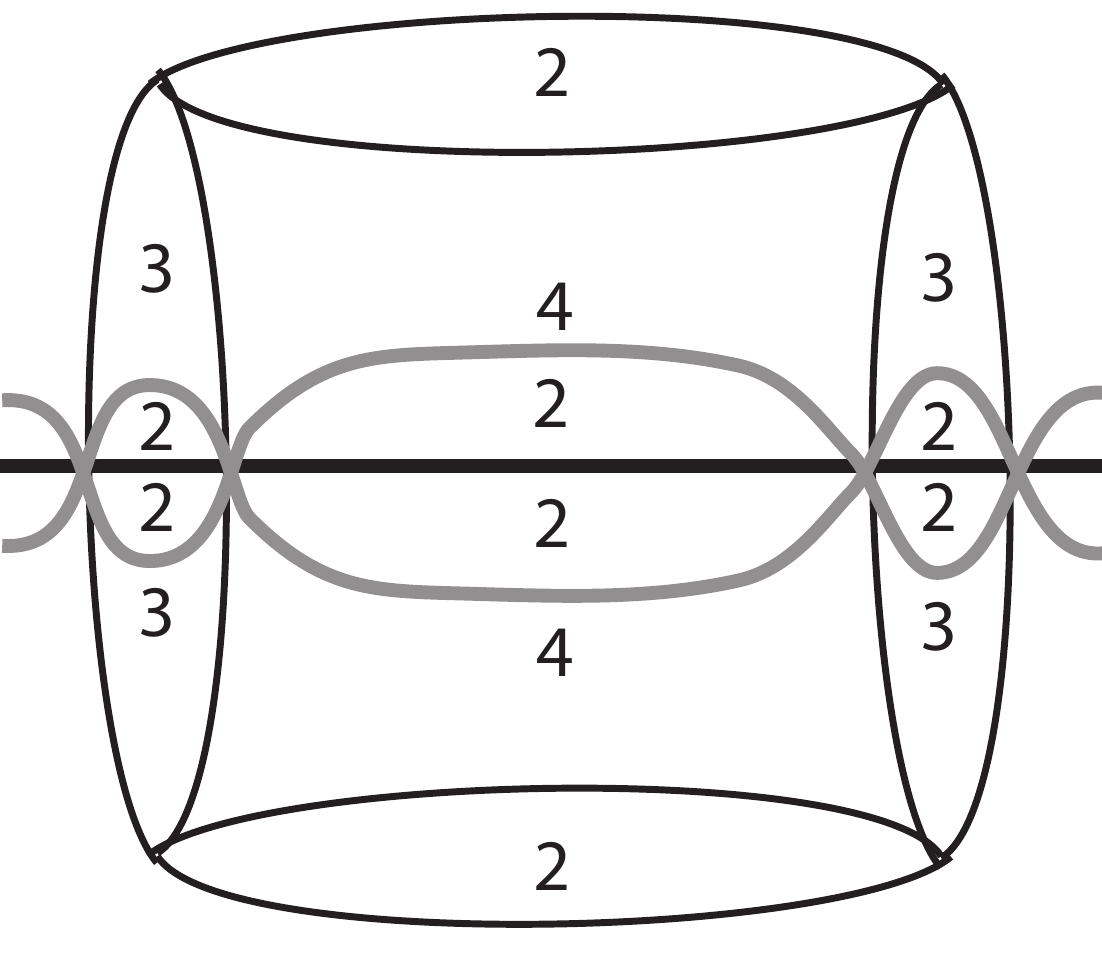}}
		\caption{}
		\label{fig:4squarewithsinglestranddoubled}
	\end{subfigure}
	\begin{subfigure}{.3\textwidth}
		\centering
		\scalebox{.3}{\includegraphics{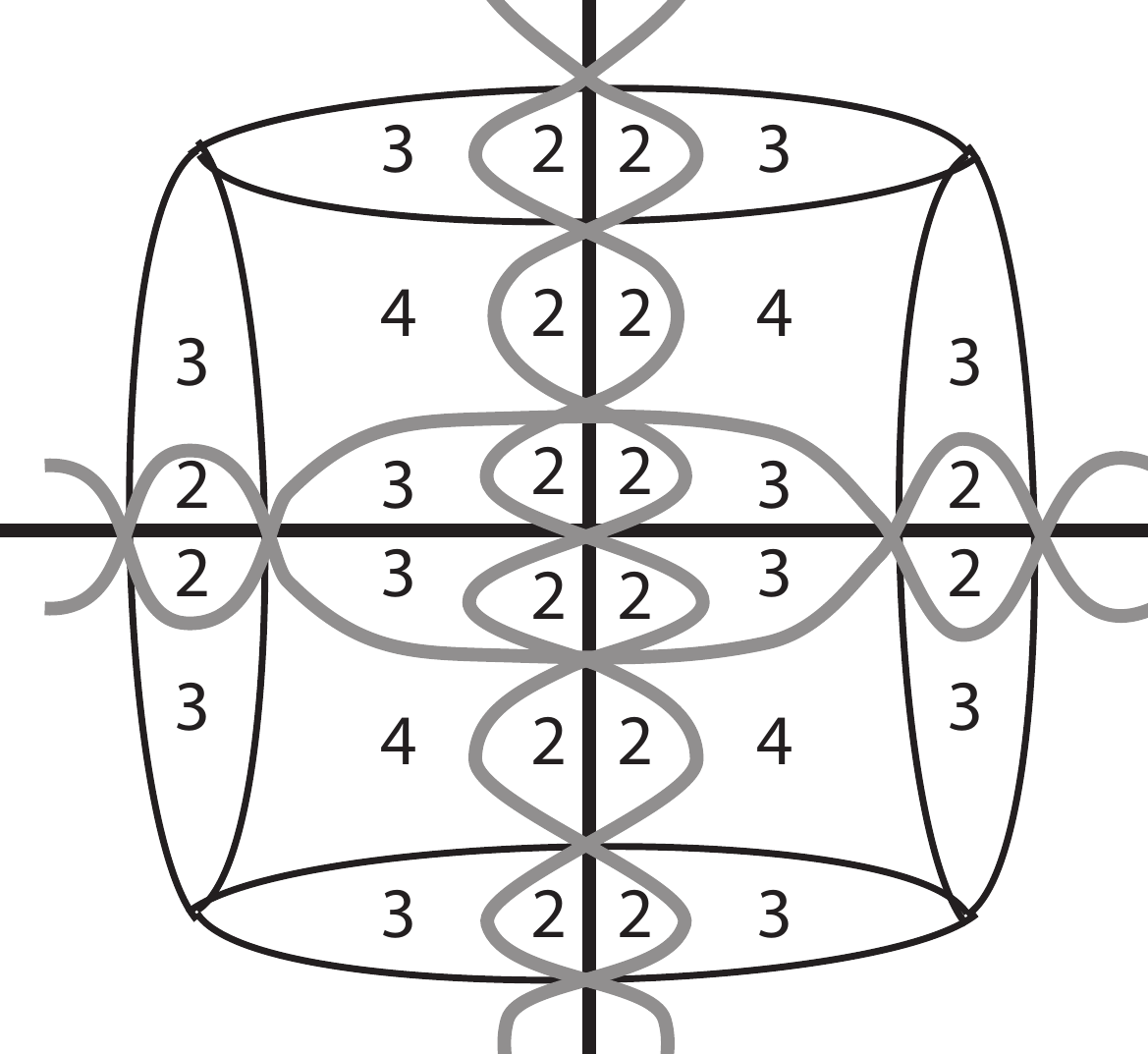}}
		\caption{}
		\label{fig:4squarewithcrossingdoubled}
	\end{subfigure}
	\caption{(a) Corner of the original knot. (b) Single strand of the original knot.  (c) Crossing of the original knot.}
	\label{fig:4knotontemplatedoubled}
\end{figure}

There is only one case of a 5-gon.  This can be replaced with monogons, bigons, and triangles using the method shown in Figure 
\ref{fig:4fixingthe5gon}.

\begin{figure}[h!]
	\begin{subfigure}[b]{0.3\textwidth}
		\centering
		\scalebox{.3}{\includegraphics{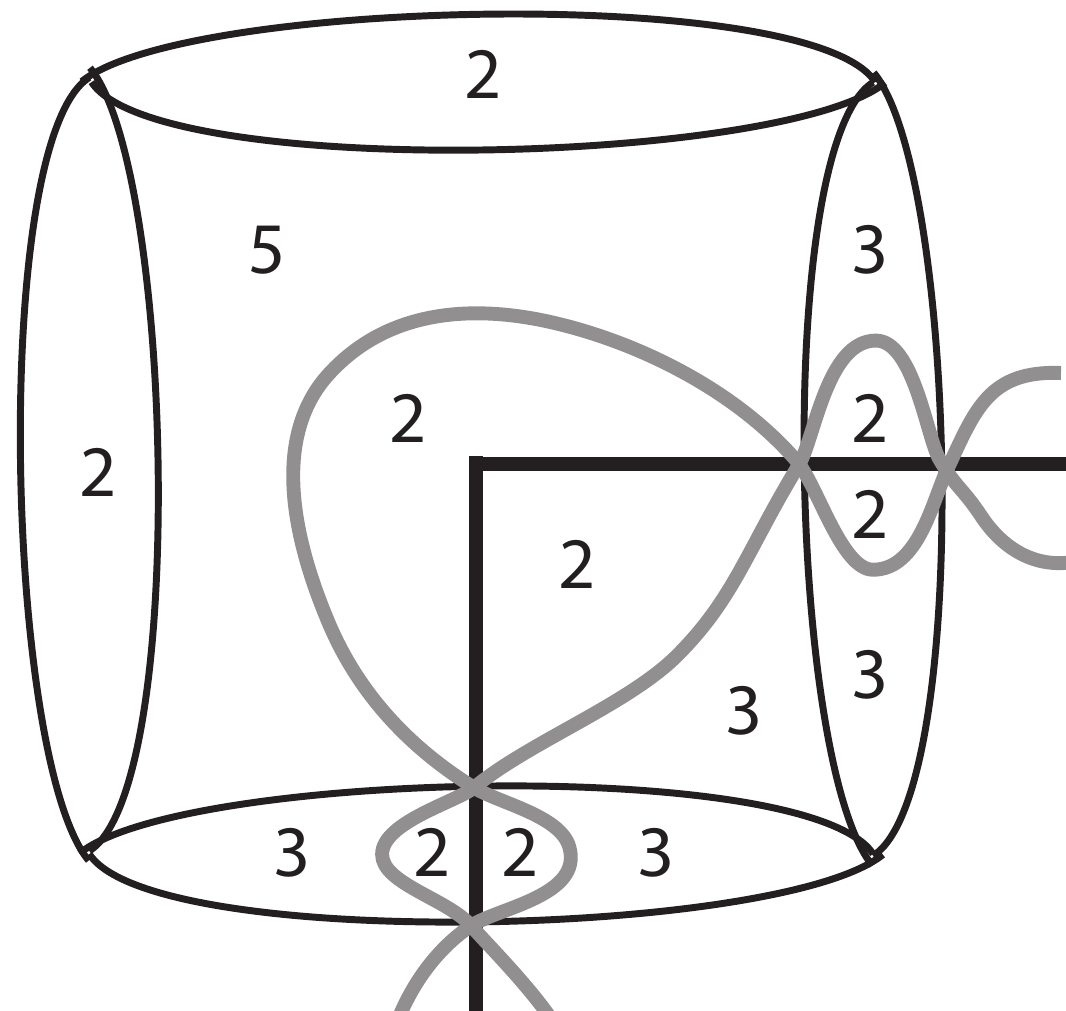}}
		\caption{}
		\label{fig:4fixing5gon1} 
	\end{subfigure}
	\begin{subfigure}[b]{0.3\textwidth}
		\centering
		\scalebox{.3}{\includegraphics{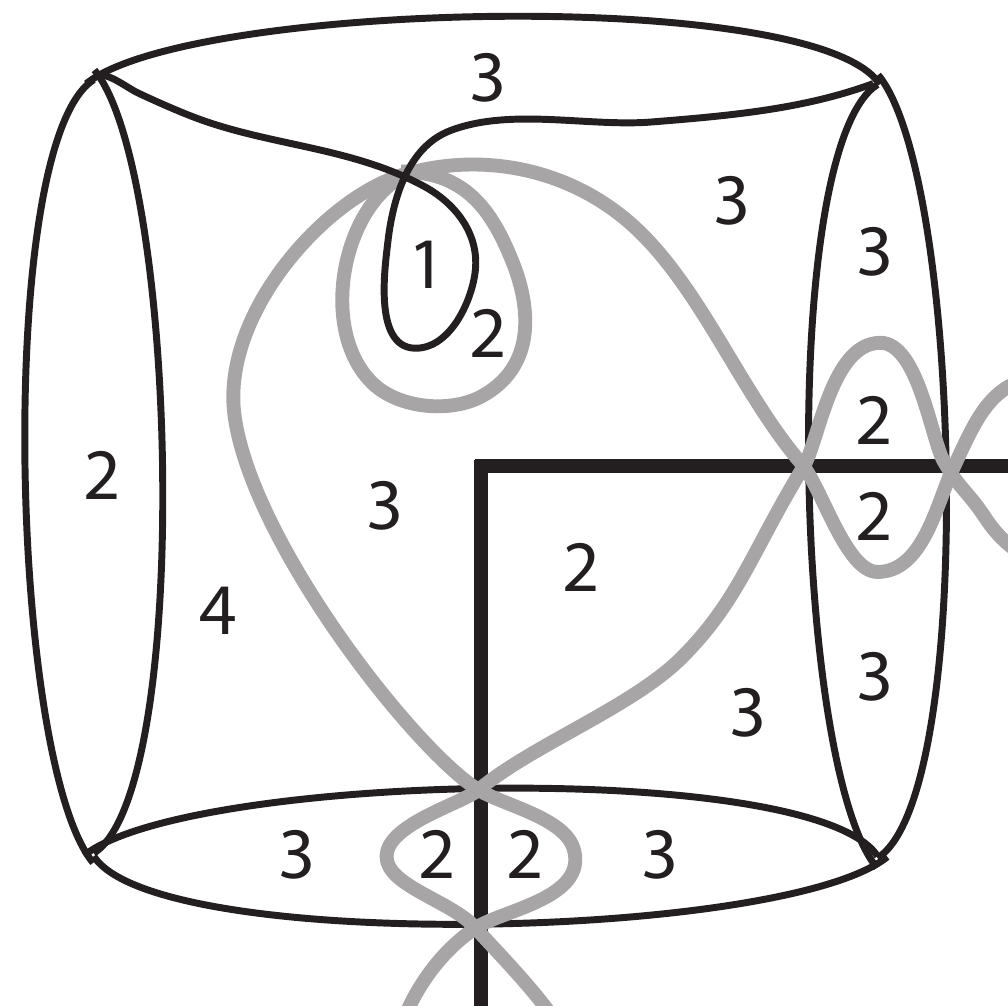}}
		\caption{}
		\label{fig:4fixing5gon2} 
	\end{subfigure}
\caption{(a) Two doubled uncrossing strands, creating a 5-gon.  (b) Adding an additional trivial 4-crossing to eliminate the 5-gon.}
\label{fig:4fixingthe5gon} 
\end{figure}

Now either $P'$ and $P''$ will become composed with each other in the process of the Reidemeister II move perturbations, or they will remain 
separate. If they become composed with each other, we need only compose the doubling knot with 
the template knot and the doubling knot with the original knot.  If not, we must compose a 
doubling knot with the original knot and 
each doubling knot with the template knot.  Compose the original knot with a doubling knot as in Figure \ref{fig:4composewithdouble} 
(note that we may need to do this twice), and 
compose the doubling knot with the template knot as in Figure \ref{fig:4composewithtemplate}.  Notice that we need one square as in 
Figure \ref{fig:4squarewithsinglestranddoubled} and one bigon to accomplish this.  This can be accomplished by choosing a fine enough grid.

\begin{figure}[h!]
	\begin{subfigure}[b]{0.3\textwidth}
		\centering
		\scalebox{.3}{\includegraphics{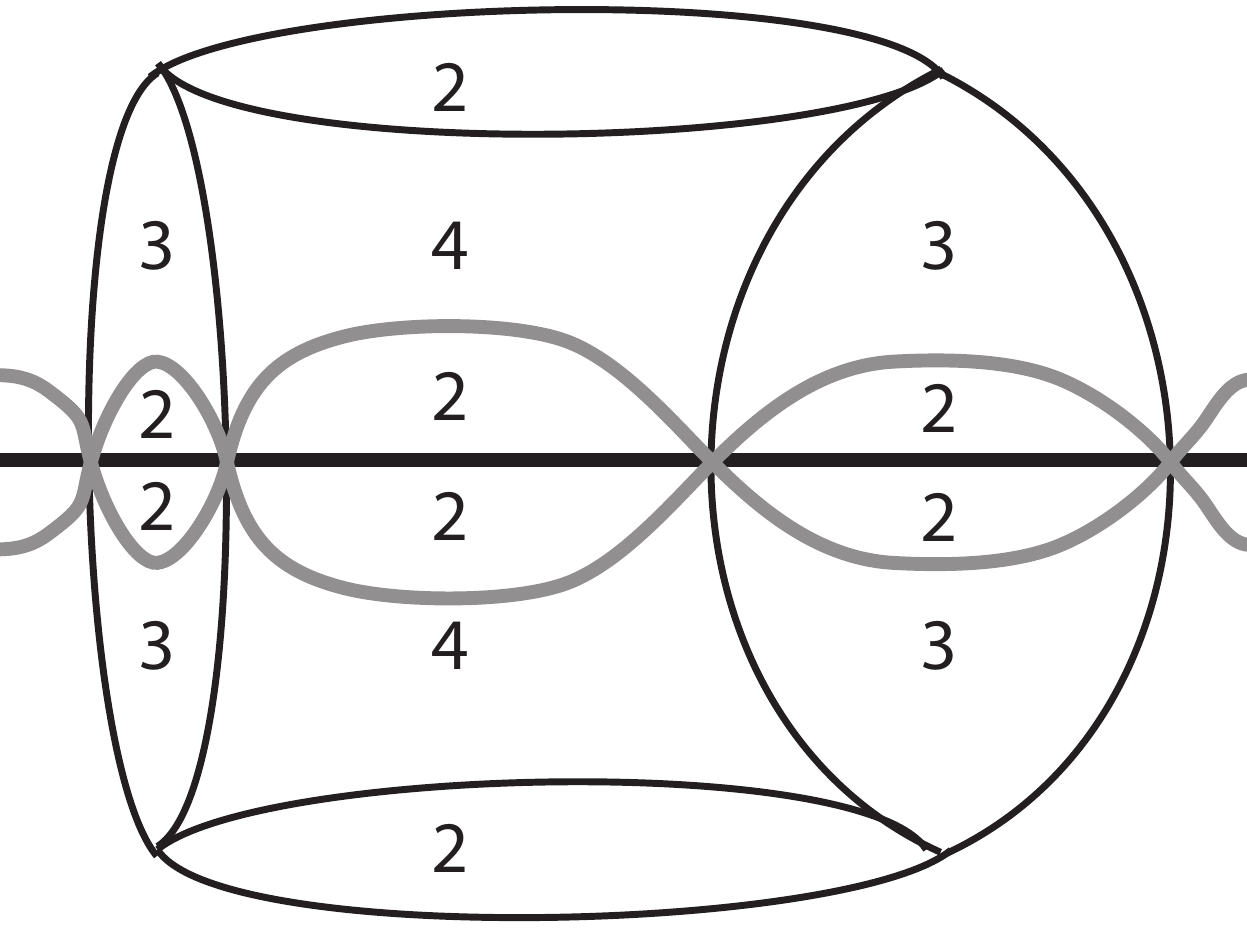}}
		\caption{}
		\label{fig:4composewithdouble1} 
	\end{subfigure}
	\begin{subfigure}[b]{0.3\textwidth}
		\centering
		\scalebox{.3}{\includegraphics{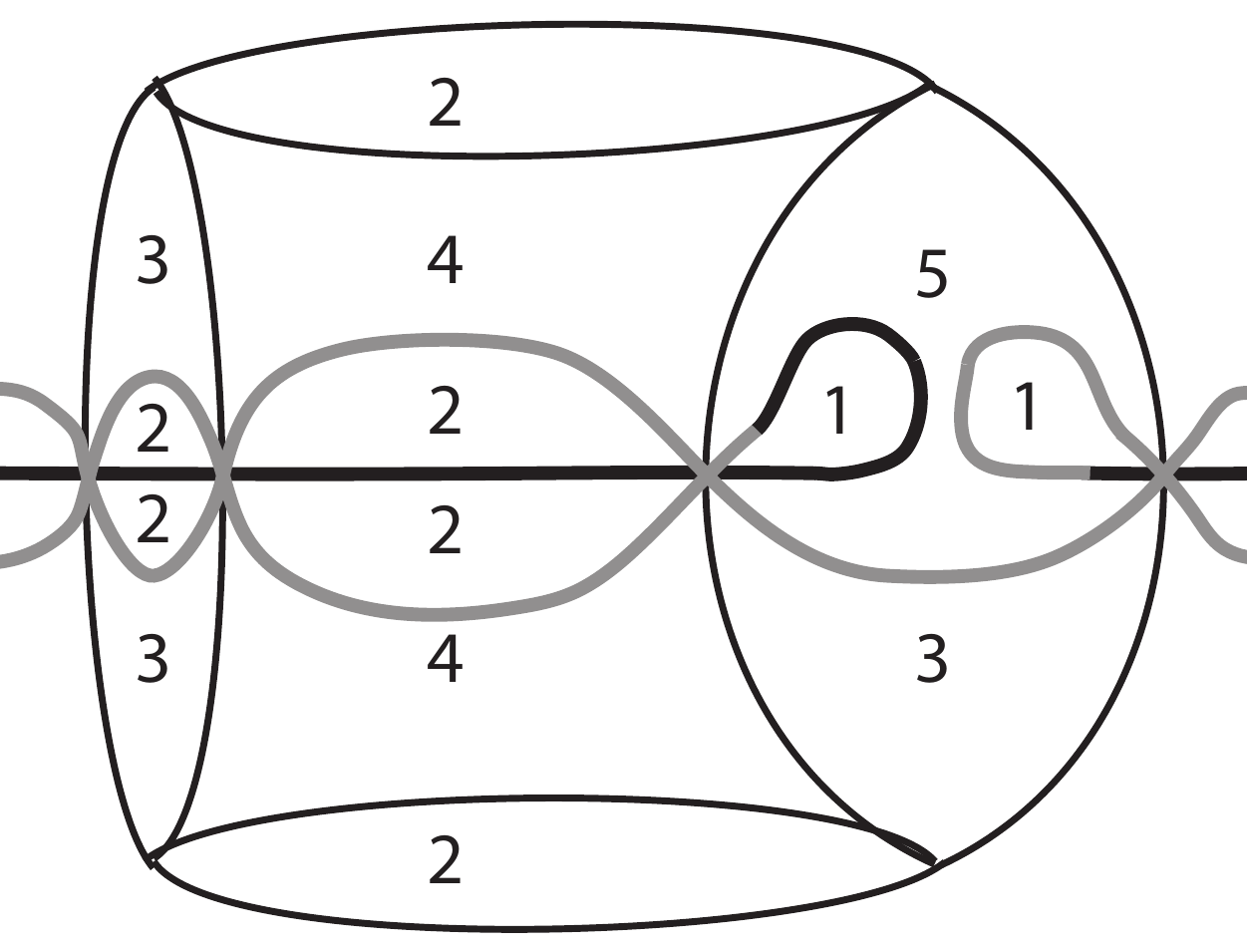}}
		\caption{}
		\label{fig:4composewithdouble2} 
	\end{subfigure}
	\begin{subfigure}[b]{0.3\textwidth}
		\centering
		\scalebox{.3}{\includegraphics{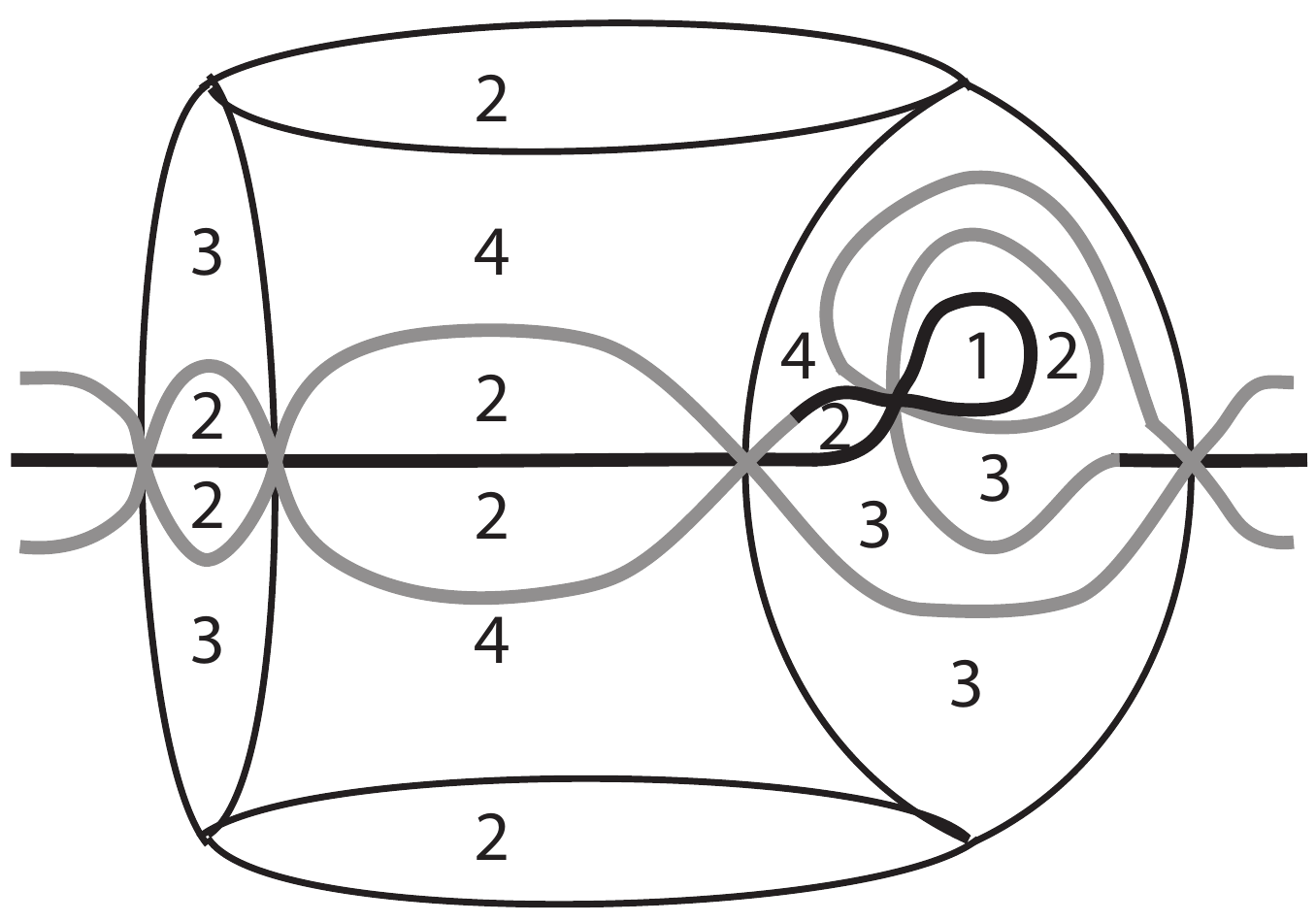}}
		\caption{}
		\label{fig:4composewithdouble3} 
	\end{subfigure}
\caption{Composition with a doubling knot. (a) Beginning projection. (b) Composition of the original knot and the doubling knot.  
		(c) Trivial triple-crossing to fix the 5-gon.}
\label{fig:4composewithdouble} 
\end{figure}

\begin{figure}[h!]
	\begin{subfigure}[b]{0.3\textwidth}
		\centering
		\scalebox{.3}{\includegraphics{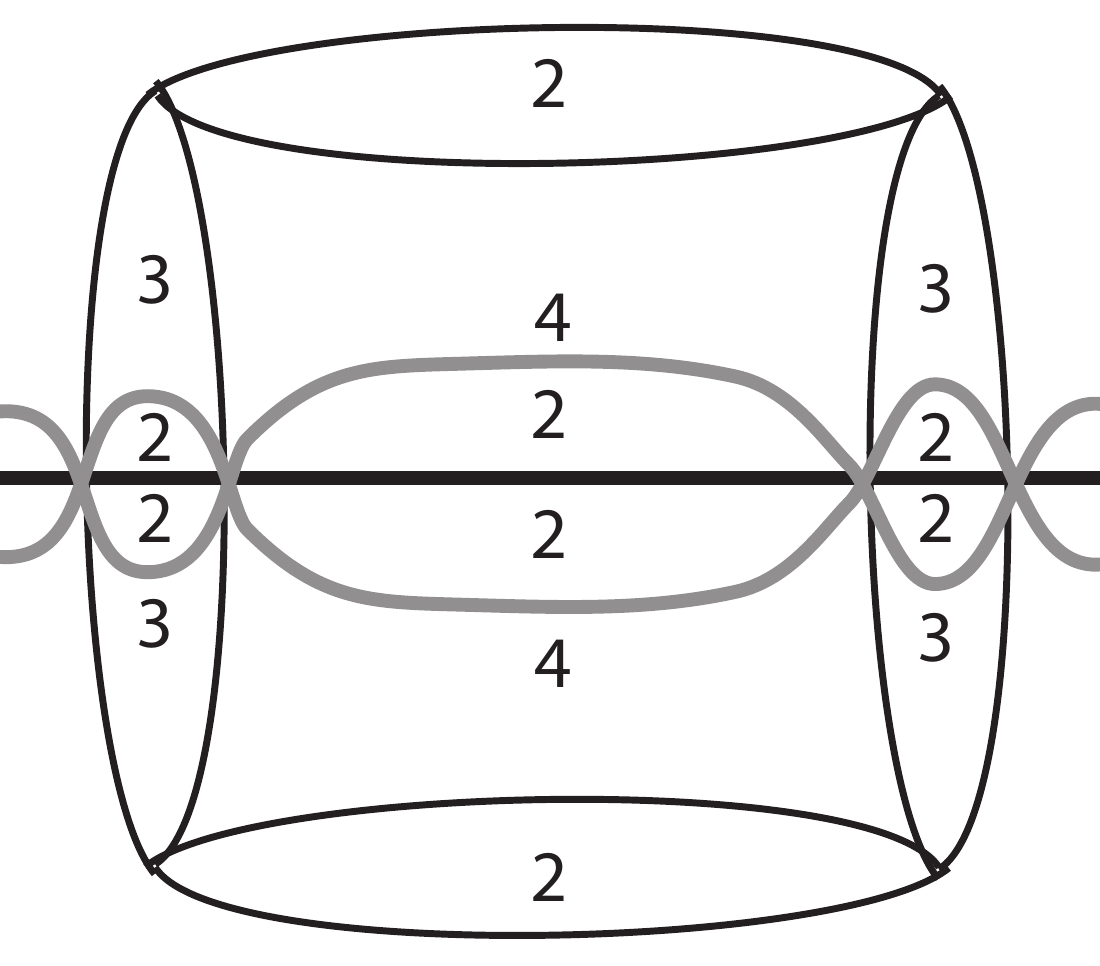}}
		\caption{}
		\label{fig:4composewithtemplate1} 
	\end{subfigure}
	\begin{subfigure}[b]{0.3\textwidth}
		\centering
		\scalebox{.3}{\includegraphics{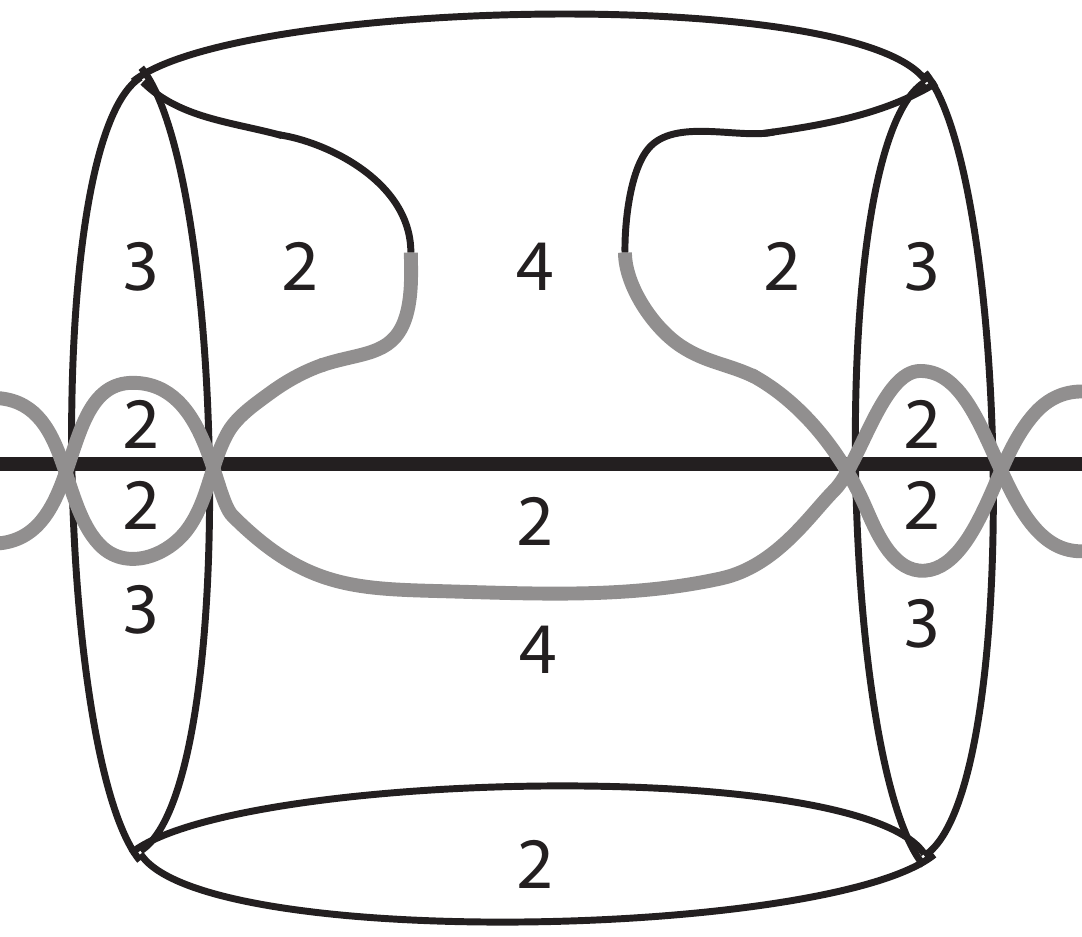}}
		\caption{}
		\label{fig:4composewithtemplate2} 
	\end{subfigure}
\caption{Composition with the template knot.  (a) Beginning projection.  (b) Composition of the doubling knot and the template.}
\label{fig:4composewithtemplate} 
\end{figure}
\end{proof}

This proof concludes the proof of the main theorem.


\section{Conclusions}

From this we can immediately see that $(1, 2, 3, 4)$ is universal for link projections as well.  We need only compose the template knot with one link component 
and each link component with its respective doubling knots.

\begin{cor}  For all $n>2$, the sequence $(1,2,3,4)$ is universal for $n$-crossing link projections.
\end{cor}

Furthemore, this gives another proof that all knots have an $n$-crossing projection for all $n\geq 3$, which appears as Theorem 4.3. in \cite{Ad}.  

\begin{cor} Let $K$ be a knot or link.  For all $n>2$, there is an $n$-crossing projection of $K$.
\end{cor}

\noindent \textbf{Open Questions}
\begin{enumerate}
	\item Is $(1,2,3,4)$, or equivalently $(2,3,4)$, universal for regular projections of knots?
	\item Are there any sequences of three integers which are universal for all $n$-crossing projections of knots?
\end{enumerate}


\begin{landscape}
$\\ \\ \\ \\ \\ \\$
	\begin{table}[ht]
	\begin{center}
	\setlength{\tabcolsep}{2pt}
	\begin{tabular}{|cccccccccccccccccccccccccc|}
	\hline &&&&&&&&&&&&&&&&&&&&&&&&& \\
		$(2n+1)C$ &$\to$ &$(2n+1)C'$ &$\to$ &$2B'$ &$\to$ &$2B$ &$\to$ &$(2n-1)C$ &$\to$ &$\dots$ &$\to$& 
						&$3C$ &$\to$ &$3C'$ &$\to$ &$(2(n-1))B'$ &$\to$ &$(2(n-1))B$ &$\to$ &$1C$ &$\to$ &$1C'$ &$\to$ &\\
		$(2n+1)B $&$\to $&$(2n+1)B'$ &$\to$ &$2A' $&$\to$ &$2A$ &$\to$ &$(2n-1)B $&$\to$ &$\dots$ &$\to$ &
						&$3B $&$\to $&$3B'$ &$\to $&$(2n)A'$ &$\to $&$(2n)A$ &$\to $&$1B$ &$\to$ &$1B'$&$\to$& \\
		$(2n+1)A$ &$\to$ &$(2n+1)A'$ &$\to$ &$2C'$ &$\to$ &$2C$ &$\to $&$(2n-1)A$ &$\to$ &$\dots$ &$\to$& 
						&$3A$ &$\to$ &$3A'$ &$\to$ &$(2n)C'$ &$\to$ &$(2n)C $&$\to$ &$1A $&$\to$ &$1A'$& $\to$ &$(2n+1)C$\\
		&&&&&&&&&&&&&&&&&&&&&&&&&\\ 
		\hline 
	\end{tabular}
	\end{center}
	\caption{}
	\label{table:3crossingsequence}
	\end{table}
	\\ \\ \\ \\
	
	\begin{table}[ht]
	\begin{center}
	\setlength{\tabcolsep}{2pt}
	\begin{tabular}{|cccccccccccccccccccc|}
	\hline &&&&&&&&&&&&&&&&&&& \\
		$1B$&$\to$		&$1B'$&$\to$	&$(2n+2)A'$	&$\to$	&$(2n+2)A$&$\to$	
				&$4B$&$\to$	&$4B'$&$\to$	&$(2n-1)A'$&$\to$	&$(2n-1)A$&$\to$ &&&& \\
		$5B$&$\to$		&$5B'$&$\to$	&$(2n-2)A'$	&$\to$	&$(2n-2)A$&$\to$	
				&$8B$&$\to$	&$8B'$&$\to$	&$(2n-5)A'$&$\to$	&$(2n-5)A$&$\to$ & &&& \\
		$\vdots$&		&$\vdots$&	&$\vdots$&		&$\vdots$&	
				&$\vdots$&		&$\vdots$&	&$\vdots$&		&$\vdots$&&&&&\\
		$(2n-1)B$&$\to$		&$(2n-1)B'$&$\to$	&$4A'$	&$\to$	&$4A$&$\to$	
				&$(2n+2)B$&$\to$	&$(2n+2)B'$&$\to$	&$1A'$&$\to$	&$1A$&$\to$ &&&& \\
		$(2n+3)B$&$\to$	&$(2n+3)B'$&$\to$	&		&	&&&&&&&&&&&&&&\\
		&&&&&&&&&&&&&&&&&&&\\
					&&&&$2A$&$\to$	&$2A'$	&$\to$	&$(2n+1)B'$&$\to$	&$(2n+1)B$ &$\to$	
							&$3A$&$\to$	&$3A'$&$\to$	&$(2n)B'$&$\to$	&$(2n)B$&$\to$ \\
					&&&&$6A$&$\to$	&$6A'$&$\to$	&$(2n-3)B'$&$\to$	&$(2n-3)B$&$\to$	
							&$7A$&$\to$	&$7A'$&$\to$	&$(2n-4)B'$&$\to$	&$(2n-4)B$&$\to$ \\
					&&&&$\vdots$&		&$\vdots$&	&$\vdots$&		&$\vdots$&   
							&$\vdots$&		&$\vdots$&	&$\vdots$&		&$\vdots$&\\
					&&&&$(2n)A$&$\to$	&$(2n)A'$&$\to$	&$3B'$&$\to$	&$2B$&$\to$	
							&$(2n+1)A$&$\to$	&$(2n+1)A'$&$\to$	&$2B'$&$\to$	&$2B$&$\to$ \\
					&&&&&&&&&&&&&&&&&&&\\
		$(2n+3)A'$&$\to$	&$(2n+3)A$&$\to$	&$1B$.&&&&&&&&&&&&&&&\\
		&&&&&&&&&&&&&&&&&&&\\ 
		\hline 
	\end{tabular}
	\end{center}
	\caption{}
	\label{table:4crossingsequence}
	\end{table}
\end{landscape}



\begin{thebibliography}{0}

\bibitem{Ad} C. Adams, \emph{Triple Crossing Number of Knots and Links}, preprint at arXiv:1207.7332.

\bibitem{AST} C. Adams, R. Shinjo, K. Tanaka, \emph{Complementary Regions of Knot and Link Diagrams}, Annals of Combinatorics, \textbf{15.4} (2011) 549-563.




\end{thebibliography}
\end {document}